\documentclass[12pt,a4]{amsart}
\usepackage{amsmath}
\usepackage{amssymb}
\usepackage{amsthm}
\usepackage{enumitem}
\usepackage{url}
\usepackage{soul}
\usepackage{times}
\usepackage[T1]{fontenc}
\usepackage[export]{adjustbox}

\usepackage{color}
\usepackage{dsfont}
\usepackage{epsfig}
\usepackage[hidelinks]{hyperref}
\usepackage{setspace}
\usepackage{epstopdf}
\usepackage[authoryear,round,sort]{natbib}
\usepackage{comment}
\usepackage{booktabs}
\usepackage{bm}
\usepackage{float}
\usepackage[linesnumbered,ruled,vlined]{algorithm2e}
\SetKwInput{KwInput}{Input}                
\SetKwInput{KwOutput}{Output} 
\usepackage{stmaryrd}
\usepackage{subcaption,graphicx}
\usepackage{float}
\numberwithin{equation}{section}

\usepackage[font=small,labelfont=bf]{caption}
\DeclareCaptionFont{tiny}{\tiny}
\newcommand{\norm}[1]{\left\lVert#1\right\rVert}
\usepackage{geometry}
\theoremstyle{plain}

\newtheorem{theorem}{Theorem}[section]
\newtheorem{lemma}{Lemma}[section]
\newtheorem{corollary}{Corollary}[section]

\theoremstyle{definition}

\newtheorem{assumption}{Assumption}[section]

\theoremstyle{remark}
\newtheorem{rk}{Remark}[section]
\expandafter\let\expandafter\oldproof\csname\string\proof\endcsname
\let\oldendproof\endproof
\renewenvironment{proof}[1][\proofname]{%
  \oldproof[\noindent\textbf{#1.} ]%
}{\oldendproof}
\newcommand{\1}{\mathds{1}}

\newcommand{\mcH}{\mathcal{H}}

\newcommand{\E}{\mathbb{E}}
\newcommand{\mP}{\mathbb{P}}
\newcommand{\be}{\begin{equation}}
\newcommand{\ee}{\end{equation}}
\newcommand{\by}{\begin{eqnarray*}}
\newcommand{\ey}{\end{eqnarray*}}

\renewcommand{\leq}{\leqslant}
\renewcommand{\geq}{\geqslant}
\usepackage{xcolor}
\definecolor{dark-red}{rgb}{0.4,0.15,0.15}
\definecolor{dark-blue}{rgb}{0.15,0.15,0.4}
\definecolor{medium-blue}{rgb}{0,0,0.5}
\hypersetup{
    colorlinks, linkcolor={red},
    citecolor={blue}, urlcolor={blue}
}
\allowdisplaybreaks

\begin{document}
\title[Landscape modification meets spin systems]{Landscape modification meets spin systems: from torpid to rapid mixing, tunneling and annealing in the low-temperature regime}
\author{Michael C.H. Choi}
\address{Department of Statistics and Data Science and Yale-NUS College, National University of Singapore, Singapore}
\email{mchchoi@nus.edu.sg}

\date{\today}
\maketitle

\begin{abstract}
	
	Given a target Gibbs distribution $\pi^0_{\beta} \propto e^{-\beta \mcH}$ to sample from in the low-temperature regime on $\Sigma_N := \{-1,+1\}^N$, in this paper we propose and analyze Metropolis dynamics that instead target an alternative distribution $\pi^{f}_{\alpha,c,1/\beta} \propto e^{-\mcH^{f}_{\alpha,c,1/\beta}}$, where $\mcH^{f}_{\alpha,c,1/\beta}$ is a transformed Hamiltonian whose landscape is suitably modified and controlled by the parameters $f,\alpha,c$ and $\beta$ and shares the same set of stationary points as $\mcH$. With appropriate tuning of these parameters, the major advantage of the proposed Metropolis dynamics on the modified landscape is that it enjoys an $\mathcal{O}(1)$ critical height while its stationary distribution $\pi^{f}_{\alpha,c,1/\beta}$ maintains close proximity with the original target $\pi^0_{\beta}$ in the low-temperature. We prove rapid mixing and tunneling on the modified landscape with polynomial dependence on the system size $N$ and the inverse temperature $\beta$, while the original Metropolis dynamics mixes torpidly with exponential dependence on the critical height and $\beta$. In the setting of simulated annealing, we prove its long-time convergence under a power-law cooling schedule that is faster than the typical logarithmic cooling in the classical setup.
	
	We illustrate our results on a host of models including the Ising model on various deterministic and random graphs as well as  Derrida's Random Energy Model. In these applications, (with high probability) the original Metropolis dynamics mixes torpidly while the proposed dynamics on the modified landscape mixes rapidly with polynomial dependence on both $\beta$ and $N$ and find the approximate ground state provably in $\mathcal{O}(N^4)$ time. This paper highlights a novel use of the geometry and structure of the landscape, in particular the global minimum value of $\mcH$ or its estimate, to the design of accelerated samplers or optimizers.
		
	\smallskip
	
	
\end{abstract}

\tableofcontents


\section{Introduction}\label{sec:intro}

In this paper, we consider spin systems on the discrete cube $\Sigma_N := \{-1,+1\}^N$ with a given Hamiltonian function $\mcH$ and $N \in \mathbb{N}$. The primary objective is to approximately sample from the Gibbs distribution $\pi^0_{\beta} \propto e^{-\beta \mcH}$ when the temperature $T = 1/\beta$ is small using the classical Metropolis dynamics and its variants. The problem of sampling in the low-temperature regime appears in a variety of application domains ranging from molecular dynamics \cite{LRS10}, Bayesian inference \cite{RR04,Robert2004} to computational physics \cite{DM13}. In addition, this setup is also utilized in stochastic optimization \cite{C04} as the Gibbs distribution $\pi^0_{\beta}$ in the low-temperature regime concentrates around the set of global minima of $\mcH$. We are interested in characterizing the time complexity, or the so-called mixing time of the classical Metropolis dynamics, as it gives practical guidance in the time it takes to simulate the chain until one approximately samples from $\pi^0_{\beta}$ or obtains an approximate minimizer of $\mcH$. Throughout this paper, we say that a mixing or tunneling time parameter $t$ is a \textbf{torpid} parameter if $t$ is at least exponential in $N^{\varepsilon}$ for some $\varepsilon > 0$. On the other hand, we say that $t$ is a \textbf{rapid} parameter if $t$ is at most polynomial in the system size $N$. Note that the notion of torpid versus rapid mixing has been studied intensively in the literature in a wide range of sophisticated statistical physics models (such as the Potts model, the Blume-Emery-Griffiths model or the Generalized Random Energy Model) and advanced Markov chain Monte Carlo algorithms (such as parallel tempering or the swapping algorithm, simulated tempering, Swendsen-Wang dynamics and the equi-energy sampler), see for example \cite{EKLV14, MZ03,WSH09a,WSH09b, BCFKTVV99, EL09,GJ99,BR16} and the references therein. 

While the classical Metropolis-Hastings algorithm is known to mix efficiently on unimodal target distribution in the Euclidean space \cite{JS18}, one of the culprits that leads to torpid mixing of the Metropolis dynamics on a possibly high-dimensional and non-convex landscape specified by the target Hamiltonian $\mcH$ is that its relaxation time grows at least in the order of $e^{\beta m}$ in the low-temperature regime \cite{HS88}, where $\beta$ is the inverse temperature and $m$ is the so-called critical or saddle height of the landscape (see \eqref{eq:m} below for the definition of $m$). In many spin models of interest, $m$ grows at least in the order of the system size $N$. Coupled with an exponential dependence on $\beta$, this consequently leads to torpid total variation mixing and relaxation time in the low-temperature setting. For example, in the ferromagnetic Ising model defined on an expander graph, $m$ grows at least in the order of $N$ \cite{DHJN17}. As another example, with high probability $m$ is shown to grow at least in the order of $N$ in Derrida's Random Energy Model in Section \ref{subsec:REM}.

To possibly overcome this bottleneck of exponential dependence on both $\beta$ and $m$ in these mixing and tunneling time parameters, one natural approach is to transform the landscape suitably to hopefully reduce the magnitude of $m$. Originated from the study of overdamped Langevin diffusion with state-dependent noise \cite{FQG97}, it is shown in \cite{C21} that the effect of state-dependent noise is de-facto equivalent to altering the target Hamiltonian from $\mcH$ to
\begin{align*}
	\mathcal{H}^{f}_{\alpha,c,1/\beta}(\sigma) := \int^{\mcH(\sigma)}_{\mcH_{*}} \dfrac{1}{\alpha f((u-c)_+) + 1/\beta}\, du,
\end{align*}
where $\sigma \in \Sigma_N$, $\mcH_{*} := \min \mcH$ and for any real number $a$, we use $a_+ := \max\{a,0\}$ to denote the non-negative part. We assume that the parameters $f,\alpha,c$ are chosen according to Assumption \ref{assump:fc} below. With appropriate tuning of these parameters, one significant advantage of targeting $\mcH^f_{\alpha,c,1/\beta}$ is that the relaxation time of the Metropolis dynamics on the modified landscape is at most of the order of $N^3$ (see Theorem \ref{thm:main} item \eqref{it:rapidrelax} below) for any $\mcH$ on $\Sigma_N$ that consequently gives rise to rapid mixing. The crux to this proof is that upon choosing $f(x) = f_2(x) := x^2$, $\alpha = \beta$, the modified Hamiltonian is given by
$$\mcH^{f_2}_{\beta,c,1/\beta}(\sigma) = \beta (\mcH(\sigma) \wedge c - \mcH_{*}) + \arctan(\beta(\mcH(\sigma)-c)_+),$$
where we write $a \wedge b = \min\{a,b\}$ for $a,b \in \mathbb{R}$. If the saddle height is attained at $\mcH^{f_2}_{\beta,c,1/\beta}(\sigma) - \mcH^{f_2}_{\beta,c,1/\beta}(\eta)$ and $c$ is chosen to be close enough to $\mcH_*$, the linear term in $\beta$ is cancelled and only the $\arctan$ term remains which can be upper bounded by $\pi/2$. The rest of the work is to ensure that the prefactor in front of the exponential term in the relaxation time upper bound is at most polynomial in both $N$ and $\beta$.

While there is perhaps the benefit of reduced critical height by running algorithms on the modified landscape, we are nonetheless introducing bias to the dynamics as we are afterall targeting $\pi^f_{\alpha,c,1/\beta}$ instead of the original Gibbs distribution $\pi^0_{\beta}$ in the Metropolis dynamics. However, when the temperature is low enough, the total variation distance between $\pi^{f}_{\alpha,c,1/\beta}$ and $\pi^0_{\beta}$ is small as both distributions can be shown to concentrate on the set of global minima of $\mcH$, and hence $\pi^{f}_{\alpha,c,1/\beta}$ is a reasonably accurate proxy or surrogate of $\pi^0_{\beta}$.

We proceed to give an informal presentation of the main results below:

\begin{theorem}[Informal presentation of Theorem \ref{thm:main}]
	Assume that $\mcH$ has a unique global minimum, and the parameters $f = f_2$, $\alpha = \beta$ and $c$ are tuned appropriately. For low enough temperature with
	$$\beta = \dfrac{1}{c - \mcH_*} \Theta(N),$$
	\begin{enumerate}
		\item(Torpid mixing and tunneling of the original Metropolis dynamics)
		The mixing and tunneling time parameters $t^0$ of the original Metropolis dynamics that targets $\pi^0_{\beta}$ are torpid, where 
		$$t^0 \in \bigg\{t^0_{rel}, t^0_{mix}, \sup_{\sigma \in \mathcal{LM}} \mathbb{E}_{\sigma}(\tau^0_{\sigma_*})\bigg\}$$ 
		can be any of the relaxation time $t^0_{rel}$, total variation mixing time $t^0_{mix}$ or the mean tunneling time to the global minimum $\sigma_*$ initiated from the worst starting state among the local minima $\sup_{\sigma \in \mathcal{LM}} \mathbb{E}_{\sigma}(\tau^0_{\sigma_*})$.
		
		\item(Rapid mixing and tunneling of the proposed Metropolis dynamics on the modified landscape)
		The mixing and tunneling time parameters $t^{f_2}$ of the Metropolis dynamics that targets $\pi^{f_2}_{\beta,c,1/\beta}$ are rapid, where
		$$t^{f_2} \in \bigg\{t^{f_2}_{rel}, t^{f_2}_{mix}, \sup_{\sigma \in \mathcal{LM}} \mathbb{E}_{\sigma}(\tau^{f_2}_{\sigma_*})\bigg\}$$ 
		can be any of the relaxation time $t^{f_2}_{rel}$, total variation mixing time $t^{f_2}_{mix}$ or the mean tunneling time to the global minimum initiated from the worst starting state among the local minima $\sup_{\sigma \in \mathcal{LM}} \mathbb{E}_{\sigma}(\tau^{f_2}_{\sigma_*})$. Note that in the definition of $t^{f_2}_{mix}$ below in Section \ref{sec:spinlandmod} we are measuring the total variation distance between the transition semigroup of the Metropolis dynamics with $\pi^0_{\beta}$ instead of $\pi^{f_2}_{\beta,c,1/\beta}$.
	\end{enumerate}
\end{theorem}

We refer readers to Section \ref{sec:spinlandmod} for the precise definition of the above mixing and tunneling time parameters, and to Section \ref{subsec:asympnotations} where we introduce the asymptotic notations. In the second main result, we prove long-time convergence of simulated annealing on the modified landscape using a power-law annealing schedule:

\begin{theorem}[Informal presentation of Theorem \ref{thm:simanneal}]
	Assume that $\mcH$ has a unique global minimum, and the parameters $f = f_2$, $\alpha_t = \beta_t$ and $c$ are tuned appropriately. Let $a \in (0,1)$ and for $t \geq 0$ under the following power-law cooling schedule,
	$$\beta_t = t^a,$$
	the simulated annealing chain converges to the ground state of $\mcH$ as $t \to \infty$.
\end{theorem}

As an application of the main results, we showcase the effect of landscape modification for sampling in the low-temperature in Derrida's Random Energy Model:

\begin{corollary}[Informal presentation of Corollary \ref{cor:REM}]
	Consider Derrida's Random Energy Model with Hamiltonian $\mcH(\sigma) = - \sqrt{N} X_{\sigma}$, where $(X_{\sigma})$ is an i.i.d. family of normal random variables. By taking $f = f_2$, $\alpha = \beta$ and $c = -N \sqrt{2 \ln 2} + \frac{N^{1/4}}{4}$, with high probability at low enough temperature we have as $N \to \infty$,
	the mixing and tunneling time parameters $t^{f_2}$ of the Metropolis dynamics that targets $\pi^{f_2}_{\beta,c,1/\beta}$ are rapid, where
	$$t^{f_2} \in \bigg\{t^{f_2}_{rel}, t^{f_2}_{mix}, \sup_{\sigma \in \mathcal{LM}}\mathbb{E}_{\sigma}(\tau^{f_2}_{\sigma_*};~\left(X_{\sigma}\right)_{\sigma \in \Sigma_N})\bigg\}$$ 
	can be any of the relaxation time $t^{f_2}_{rel}$, total variation mixing time $t^{f_2}_{mix}$ or the mean tunneling time to the global minimum conditional on the disorder initiated from the worst starting state among the local minima $\sup_{\sigma \in \mathcal{LM}} \mathbb{E}_{\sigma}(\tau^{f_2}_{\sigma_*};~\left(X_{\sigma}\right)_{\sigma \in \Sigma_N})$.
\end{corollary}

Let us remark that the torpid mixing of the Metropolis dynamics in the Random Energy Model is also investigated in \cite{Mathieu2000}.

This paper leverages and takes advantages of the structure and information of the energy landscape, in particular $\mcH_{*}$, the global minimum value of $\mcH$, in the design and analysis of accelerated samplers in the low-temperature regime or approximate optimizers. In a similar vein, we mention Catoni's energy transformation algorithm \cite{Catoni96,Catoni98} that utilizes a lower bound of $\mcH_{*}$ in the design of the accelerated algorithm. We note that the landscape modification Metropolis dynamics can be understood as a state-dependent generalization of Catoni's algorithm when we take $f(x) = x$ and $\alpha = 1$, see the discussion in \cite{C21}. In another related line of work \cite{CZ21}, we investigate the effect of landscape modification on sampling at temperature $T = 1$ by annealing on the parameter $\alpha$. In this work however as far as the main results are concerned we shall be taking $\alpha = \beta$ which is large.

The rest of this paper is organized as follows. In Section \ref{subsec:asympnotations}, we fix some commonly used asymptotic notations. In Section \ref{sec:spinlandmod}, we formally introduce the definition of landscape modification and in particular $\mcH^f_{\alpha,c,1/\beta}$. We also recall some familiar mixing and tunneling time notions in the study of Markov chain Monte Carlo as well as the notions of critical height on both the original and the modified landscape. The two main results, with one concerning the fixed temperature case while the second one discusses the annealing case, are presented in Section \ref{sec:main}, followed by their proofs in Section \ref{sec:pfmain}. Finally, in Section \ref{sec:applications}, we apply the results to sampling in the low-temperature or finding an approximate ground state for the Ising model on the complete graph, the random $r$-regular graph, the Erd\H{o}s-R\'{e}nyi random graph and Derrida's Random Energy Model.

\subsection{Notations}\label{subsec:asympnotations}

In this subsection, we recall a few standard asymptotic notations that are used throughout the paper. Suppose that we have two real-valued functions $g,h$. We write $g(N) = \mathcal{O}(h(N))$ if there exists a positive constant $M$ such that for all large enough $N$ we have $|g(N)| \leq M h(N)$. We also write $g(N) = \Omega(h(N))$ if there exists a positive constant $M$ such that for all large enough $N$ we have $|g(N)| \geq M h(N)$. If both $g(N) = \mathcal{O}(h(N))$ and $g(N) = \Omega(h(N))$ hold, then we write $g(N) = \Theta(h(N))$. Finally, we write $g(N) \overset{N \to \infty}{\sim} h(N)$ if $\lim_{N \to \infty} g(N)/h(N) = 1$.

%
%
%
%
%

\section{Spin systems with landscape modification}\label{sec:spinlandmod}

In this paper, the primary objective is to design accelerated samplers (i.e. provably rapidly mixing) to approximately sample from the Gibbs distribution $\pi^0_{\beta} \propto e^{-\beta \mcH}$ in the low-temperature regime when the temperature $T = 1/\beta$ is small. While there are various advanced stochastic sampling algorithms to perform such task (as mentioned in the Introduction in Section \ref{sec:intro}), we are particularly interested in the classical Metropolis dynamics. Consider a simple random walk proposal with transition matrix $P^{SRW} = (P^{SRW}(\eta,\sigma))_{\eta,\sigma \in \Sigma_N}$ given by
\begin{align*}
	P^{SRW}(\eta,\sigma) &:= \dfrac{1}{N} \1_{\{\textrm{there exists } i \textrm{ such that } \eta(i) = -\sigma(i) \textrm{ and } \eta(j) = \sigma(j) \textrm{ for all }j \neq i\}}.
\end{align*}
The stationary distribution of $P^{SRW}$ is the uniform distribution on $\Sigma_N$ given by $\mu^{SRW}(\sigma) := 1/2^N$. For the Metropolis dynamics, we instead target the Gibbs distribution with a landscape-modified Hamiltonian function given by, for any spin configuration $\sigma = (\sigma(i))_{i=1}^N \in \Sigma_N$,
\begin{align*}
	\mathcal{H}^{f}_{\alpha,c,1/\beta}(\sigma) = \int^{\mcH(\sigma)}_{\mcH_{*}} \dfrac{1}{\alpha f((u-c)_+) + 1/\beta}\, du,
\end{align*}
where we recall from the Introduction in Section \ref{sec:intro} that $\mcH_{*} = \min \mcH$ and $a_+ = \max\{a,0\}$ is the non-negative part of a real number $a$. We also write $\mcH_{max} := \max_{\sigma \in \Sigma_N} \mcH(\sigma)$ throughout this paper. The function $f$ describes the effect of the transformation above the threshold parameter $c$, while the parameter $\alpha$ controls the magnitude of the transformation. These parameters are chosen to satisfy the following assumption:
\begin{assumption}\label{assump:fc}
	$f : \mathbb{R}^+ \to \mathbb{R}^+$ is non-negative, non-decreasing and satisfies $f(0) = 0$. $c$ is chosen such that $\mcH_{max} \geq c \geq \mcH_{*}$. $\alpha$ is assumed to be non-negative, i.e. $\alpha \geq 0$.
\end{assumption}
In our main results in Section \ref{sec:main} below we will primarily be interested in $f(x) = f_2(x) := x^2$, the quadratic function, and some other common choices of $f$ include the linear function $f(x) = x$ and the square root function $f(x) = \sqrt{x}$. We note that $\mcH^{f_2}_{\alpha,c,1/\beta}$ can be written as
\begin{align}\label{eq:mcHf2alphabeta}
	\mcH^{f_2}_{\alpha,c,1/\beta} (\sigma) = \beta (c\wedge \mcH(\sigma)-\mcH_{*}) + \sqrt{\dfrac{\beta}{\alpha}} \arctan(\sqrt{\alpha \beta}(\mcH(\sigma)-c)_+).
\end{align}
For the explicit expressions of $\mcH^f_{\alpha,\beta,1/\beta}$ for other choices of $f$, we refer readers to the calculation in \cite{C21}. We will explain the benefits of $\mcH^{f_2}_{\beta,c,1/\beta}$ over other similar transformations or modifications in Section \ref{sec:main}.

With these notations in mind, the continuized Metropolis dynamics is a Markov chain that we denote by $X^{f}_{\alpha,c,1/\beta} = (X^{f}_{\alpha,c,1/\beta}(t))_{t \geq 0}$ with infinitesimal generator $L^{f}_{\alpha,c,1/\beta} = (L^{f}_{\alpha,c,1/\beta}(\eta,\sigma))_{\eta,\sigma \in \Sigma_N}$ defined to be
\begin{align}\label{eq:Llm}
	L^{f}_{\alpha,c,1/\beta}(\eta,\sigma) := \begin{cases}  P^{SRW}(\eta,\sigma) e^{-(\mathcal{H}^{f}_{\alpha,c,1/\beta}(\sigma)-\mathcal{H}^{f}_{\alpha,c,1/\beta}(\sigma_1))_+}, &\mbox{if } \eta \neq \sigma; \\
		- \sum_{\xi: \xi \neq \eta} L^{f}_{\alpha,c,1/\beta}(\eta,\xi), & \mbox{if } \eta = \sigma, \end{cases}
\end{align}
and stationary distribution given by a generalized Gibbs distribution:
$$\pi^{f}_{\alpha,c,1/\beta}(\sigma) = \dfrac{e^{-\mcH^f_{\alpha,c,1/\beta}(\sigma)} \mu^{SRW}(\sigma)}{\sum_{\sigma \in \Sigma_N} e^{-\mcH^f_{\alpha,c,1/\beta}(\sigma)} \mu^{SRW}(\sigma)} =: \dfrac{e^{-\mcH^f_{\alpha,c,1/\beta}(\sigma)}}{Z^f_{\alpha,c,1/\beta}}\propto e^{-\mcH^f_{\alpha,c,1/\beta}(\sigma)}.$$
Note that in the special case when we take $f \equiv 0$ or $\alpha \equiv 0$, we recover the classical Gibbs distribution at inverse temperature $\beta$ and Hamiltonian $\mcH$ since $\pi^0_{\beta} = \pi^0_{\alpha,c,1/\beta} = \pi^f_{0,c,1/\beta}$. As such without overburdening our notations we write $X^0_{\beta} \equiv X^0_{\alpha,c,1/\beta}$, $L^0_{\beta} \equiv L^0_{\alpha,c,1/\beta}$ and $Z^0_{\beta} \equiv Z^0_{\alpha,c,1/\beta}$.

Before we introduce the list of mixing and tunneling time parameters of both $X^f_{\alpha,c,1/\beta}$ and $X^0_{\beta}$, we shall fix and recall some classical notions in Markov chain mixing. For any two probability measures $\mu, \nu$ on $\Sigma_N$, the total variation distance is given by $\norm{\mu - \nu}_{TV} := 1/2 \sum_{\sigma \in \Sigma_N} |\mu(\sigma) - \nu(\sigma)|$. For $A \subset \Sigma_N$, we also write $\tau_A^f := \inf\left\{t \geq 0;~X^f_{\alpha,c,1/\beta}(t) \in A\right\}$ to be the first hitting time of $A$ by $X^f_{\alpha,c,1/\beta}$, and we denote the hitting time to a single state by $\tau^f_{\sigma} = \tau^f_{\{\sigma\}}$. The usual convention of $\inf \emptyset = \infty$ applies. We say that $\sigma \in \Sigma_N$ is a local minimum of $\mcH$ if for all $i \in \llbracket N \rrbracket := \{1,2,\ldots,N\}$, we have $\mcH(\sigma^{(i)}) \geq \mcH(\sigma)$, where $\sigma^{(i)}_j = \sigma_j$ for all $j \neq i$ and $\sigma^{(i)}_i = -\sigma_i$. Intuitively, $\sigma$ is a local minimum if flipping the sign of any of the $N$ coordinates does not decrease the value of the Hamiltonian. We also endow the Hilbert space $\ell^2(\pi^f_{\alpha,c,1/\beta})$ with the usual inner product $\langle u,v \rangle_{\pi^f_{\alpha,c,1/\beta}} := \sum_{\sigma \in \Sigma_N} u(\sigma)v(\sigma) \pi^f_{\alpha,c,1/\beta}(\sigma)$, where $u, v \in \ell^2(\pi^f_{\alpha,c,1/\beta})$ are real-valued functions defined on $\Sigma_N$.

Given $\varepsilon > 0$, we introduce the following list of parameters of interest. We emphasize that the dependency on $\alpha,c,\beta$ maybe suppressed in the notations when there is no ambiguity. 
\begin{itemize}
	\item(Total variation mixing time to $\pi^0_{\beta}$ by $X^f_{\alpha,c,1/\beta}$ (resp.~$X^0_{\beta}$) on the modified (resp.~original) landscape)
	\begin{align*}
		t^f_{mix}(\varepsilon) &:= \inf\bigg\{t \geq 0;~ \sup_{\sigma \in \Sigma_N} \norm{(P^f_{\alpha,c,1/\beta})^t(\sigma,\cdot)-\pi^0_{\beta}}_{TV} \leq \varepsilon\bigg\}. \\
		t^0_{mix}(\varepsilon) &:= \inf\bigg\{t \geq 0;~ \sup_{\sigma \in \Sigma_N} \norm{(P^0_{\beta})^t(\sigma,\cdot)-\pi^0_{\beta}}_{TV} \leq \varepsilon\bigg\}.
	\end{align*}
	Note that we are \textit{not} measuring $\norm{(P^f_{\alpha,c,1/\beta})^t(\sigma,\cdot)-\pi^f_{\alpha,c,1/\beta}}_{TV}$ in the definition of $t^f_{mix}(\varepsilon)$. By convention in the literature we write $t^f_{mix} = t^f_{mix}(1/4)$ and $t^0_{mix} = t^0_{mix}(1/4)$.

	\item(Relaxation time of $X^f_{\alpha,c,1/\beta}$ (resp.~$X^0_{\beta}$) on the modified (resp.~original) landscape)
	First, we recall that the second smallest eigenvalue of $-L^f_{\alpha,c,1/\beta}$ and $-L^0_{\beta}$ are defined respectively to be
	\begin{align*}
		\lambda_2(-L^f_{\alpha,c,1/\beta}) :&= \inf_{g;~ \pi^f_{\alpha,c,1/\beta}(g) = 0} \dfrac{\langle -L^f_{\alpha,c,1/\beta}g,g \rangle_{\pi^f_{\alpha,c,1/\beta}}}{\langle g,g \rangle_{\pi^f_{\alpha,c,1/\beta}}}, \\
		\lambda_2(-L^0_{\beta}) :&= \inf_{g;~ \pi^0_{\beta}(g) = 0} \dfrac{\langle -L^0_{\beta}g,g \rangle_{\pi^0_{\beta}}}{\langle g,g \rangle_{\pi^0_{\beta}}}.
	\end{align*}
	Their respective relaxation times are then
	\begin{align*}
		t^f_{rel} :&= \dfrac{1}{\lambda_2(-L^f_{\alpha,c,1/\beta})}, \quad
		t^0_{rel} := \dfrac{1}{\lambda_2(-L^0_{\beta})}.
	\end{align*}
	
	\item(Largest mean tunneling time from a local minimum to a global minimum)
	Let $\mathcal{LM} := \{\sigma \in \Sigma_N;~ \mathcal{H}(\sigma^{(i)}) \geq \mathcal{H}(\sigma) \quad \mathrm{for}\,\, i = 1,\ldots,N\}$ be the set of local minima of $\mcH$ and suppose that $\sigma_* \in \mathcal{GM} := \{\sigma \in \Sigma_N;~\mathcal{H}(\sigma) = \mathcal{H}_*\}$ is a global minimum. We would be interested in estimating both
	\begin{align*}
		\sup_{\sigma \in \mathcal{LM}} &\mathbb{E}_{\sigma}(\tau^f_{\sigma_*}), \quad 
		\sup_{\sigma \in \mathcal{LM}} \mathbb{E}_{\sigma}(\tau^0_{\sigma_*}).
	\end{align*}

	\item(First time reaching $\mcH_*$ with high probability) We shall be comparing
	\begin{align*}
		\mathcal{T}^f(\varepsilon) = \inf\bigg\{t \geq 0;~ \inf_{\sigma \in \Sigma_N}\mP_{\sigma}(\mcH(X^f_{\alpha,c,1/\beta}(t)) = \mcH_*) &\geq 1 - \varepsilon\bigg\}, \\
		\mathcal{T}^0(\varepsilon) = \inf\bigg\{t \geq 0;~ \inf_{\sigma \in \Sigma_N}\mP_{\sigma}(\mcH(X^0_{\beta}(t)) = \mcH_*) &\geq 1 - \varepsilon\bigg\}.
	\end{align*}
\end{itemize}

We proceed to introduce two critical heights: $m^f_{\alpha,c,1/\beta}$ and $m$. These constants play an important role in understanding our main results in Section \ref{sec:main}, and in particular in bounding the spectral gap of the dynamics. A path from $\eta$ to $\sigma$ is any sequence of spin configurations starting from $\sigma_1 = \eta, \sigma_2,\ldots, \sigma_n = \sigma$ such that $P^{SRW}(\sigma_{i-1},\sigma_i) > 0$ for $i = 2,\ldots,n$. For any $\eta \neq \sigma$, such path exists as the proposal chain $P^{SRW}$ is irreducible. We write $\Gamma^{\eta,\sigma}$ to be the set of paths from $\eta$ to $\sigma$, and elements of $\Gamma^{\eta,\sigma}$ are denoted by $\gamma = (\gamma_i)_{i=1}^n$. By considering $\mcH(\sigma)$ as the elevation at $\sigma$, the highest elevation along a path $\gamma \in \Gamma^{\eta,\sigma}$ is defined as
$$\mathrm{Elev}(\gamma) := \max\{\mcH(\gamma_i);~\gamma_i \in \gamma\},$$ 
and the least highest elevation from $\eta$ to $\sigma$ is 
\begin{align}\label{eq:leasthighestelev}
	H(\eta,\sigma) := \min\{\mathrm{Elev}(\gamma);~\gamma \in \Gamma^{\eta,\sigma}\}.
\end{align}
The classical critical height $m$ is then defined to be
\begin{align}
	m = m(P^{SRW},\mcH) &:= \max_{\eta,\sigma \in \Sigma_N}\{H(\eta,\sigma) - \mcH(\eta) - \mcH(\sigma)\} + \mcH_{*}. \label{eq:m} 
\end{align} 

We also consider $\Lambda^{\eta,\sigma}$ to be the set of paths from $\eta$ to $\sigma$ by flipping the coordinates in $\{i;~\eta_i \neq \sigma_i\}$ one at a time and from left to right, and the element of $\Lambda^{\eta,\sigma}$ is denoted by $\lambda = (\lambda_i)_{i=1}^n$. Note that $n = n(\eta,\sigma) \leq N$ by definition. The highest elevation along a path $\lambda \in \Lambda^{\eta,\sigma}$ is 
$$\mathrm{Elev}^f_{\alpha,c,1/\beta}(\eta,\sigma) := \max\{\mcH^f_{\alpha,c,1/\beta}(\lambda_i);~\lambda_i \in \lambda\}.$$ 
The critical height $m^f_{\alpha,c,1/\beta}$ is defined to be
\begin{align}
	m^f_{\alpha,c,1/\beta} = m^f_{\alpha,c,1/\beta}(P^{SRW},\mcH) &:= \max_{\eta,\sigma \in \Sigma_N}\{\mathrm{Elev}^f_{\alpha,c,1/\beta}(\eta,\sigma) - \mcH^f_{\alpha,c,1/\beta}(\eta) - \mcH^f_{\alpha,c,1/\beta}(\sigma)\}. \label{eq:rch} 
\end{align} 
Note that $m^f_{\alpha,c,1/\beta}$ does not necessarily reduce to $m$ when we take $f = 0$ or $\alpha = 0$ since the way we pick the paths in $\Gamma^{\eta,\sigma}$ and $\Lambda^{\eta,\sigma}$ are different. In fact we see that $\Lambda^{\eta,\sigma} \subseteq \Gamma^{\eta,\sigma}$.

\section{Main results}\label{sec:main}

In this section, we state the main results in Theorem \ref{thm:main} and Theorem \ref{thm:simanneal}. In Theorem \ref{thm:main}, we consider the setting of fixed but small enough temperature. Thus the results are particularly useful in the context of approximately sampling from the Gibbs distribution $\pi^0_{\beta}$ in the low-temperature or stochastic optimization of the Hamiltonian function $\mcH$. On the other hand in Theorem \ref{thm:simanneal} we investigate the setting of simulated annealing which the temperature cools down according to an annealing schedule. With suitable choices of the landscape-modification parameters such as $f, c$ and $\alpha$, we prove that one can operate an accelerated annealing schedule, namely the power-law schedule $(t^{-a})_{t \geq 0}$ with $a \in (0,1)$, that converges to zero faster than the typical logarithmic annealing. We emphasize that these results hold for all target Hamiltonian functions defined on $\Sigma_N$ that satisfy certain regularity conditions. These conditions are readily checked in all models to be presented in Section \ref{sec:applications}.

The central intuition to the proofs of these two main results is that we choose $f = f_2$, $\alpha = \beta$, $c \in (\mathcal{H}_*,\min_{\sigma \in \mathcal{LM}\cap\mathcal{GM}'}\mathcal{H}(\sigma)]$ while $\mcH$ is assumed to have a unique global minimum. In this setup, the modified Hamiltonian is then $$\mcH^{f_2}_{\beta,c,1/\beta}(\sigma) = \beta (\mcH(\sigma) \wedge c - \mcH_{*}) + \arctan(\beta(\mcH(\sigma)-c)_+),$$
and since 
$$m^f_{\alpha,c,1/\beta} \leq \max_{\sigma \in \Sigma_N} \mcH^f_{\alpha,c,1/\beta}(\sigma) - \min_{\sigma \in \mathcal{LM}\cap\mathcal{GM}'}\mcH^f_{\alpha,c,1/\beta}(\sigma),$$
the linear term of $\beta$ gets cancelled out when $c \in (\mathcal{H}_*,\min_{\sigma \in \mathcal{LM}\cap\mathcal{GM}'}\mathcal{H}(\sigma)]$, so that the modified critical height $m^{f_2}_{\beta,c,1/\beta}$ can be bounded above independently of the system size $N$ owing to the $\arctan$ transformation. This observation is crucial to the proof. Another key observation is that at low enough temperature, the total variation distance between $\pi^{f_2}_{\beta,c,1/\beta}$ and $\pi^0_{\beta}$ is small as both of them concentrate around the global minimum $\mcH_{*}$. As such it is advantageous to utilize $X^{f_2}_{\beta,c,1/\beta}$ to approximately sample from  $\pi^0_{\beta}$ since this dynamics enjoys a smaller critical height $m^{f_2}_{\beta,c,1/\beta}$ than the original dynamics $X^0_{\beta}$. We now proceed to present the main results:

\begin{theorem}\label{thm:main}
	Define $\Delta := \min_{\sigma;~\mcH(\sigma) > \mcH_*}\mcH(\sigma) - \mcH_{*}$. Suppose that $\mcH$ has a unique global minimum with $\Delta , m > 0$, and we choose $f(x) = f_2(x) = x^2$, $\alpha = \beta$ and $c \in (\mathcal{H}_*,\min_{\sigma \in \mathcal{LM}\cap\mathcal{GM}'}\mathcal{H}(\sigma)]$. Given $\varepsilon > 0$, for low enough temperature such that
	$$\beta \geq \dfrac{1}{(c-\mcH_{*}) \wedge \Delta \wedge (m/4)}\left((N \ln 2) + \ln(4/\varepsilon)\right),$$
	we have the following:
	\begin{enumerate}
		\item(Torpid relaxation time with exponential dependence on $\beta$ and $m$ on the original landscape)\label{it:torpidrelax}
		$$t^0_{rel} \geq \dfrac{1}{4^N} e^{\beta m} = \Omega(4^{N+1}/\varepsilon).$$
		Moreover, for fixed $N$ 
		$$\lim_{\beta \to \infty} \dfrac{1}{\beta} \ln t^{0}_{rel} = m.$$
		
		\item(Rapid relaxation time with polynomial dependence on $N$ on the modified landscape)\label{it:rapidrelax}
		$$t^{f_2}_{rel} \leq \dfrac{N^3}{2} e^{\pi/2} = \mathcal{O}(N^3).$$
		Moreover, for fixed $N$ the relaxation time $t^{f_2}_{rel}$ is subexponential in $\beta$ in the sense that 
		$$\lim_{\beta \to \infty} \dfrac{1}{\beta} \ln t^{f_2}_{rel} = 0.$$
		
		\item(Torpid total variation mixing time with exponential dependence on $N$ using $X^0_{\beta}$)\label{it:torpidmix}
		$$t^0_{mix}(\varepsilon) \geq t^0_{rel} \ln\left(\dfrac{1}{2\varepsilon}\right)= \Omega\left(\dfrac{4^N}{\varepsilon} \ln\left(\dfrac{1}{2\varepsilon}\right)\right).$$
		
		\item(Rapid total variation mixing time with polynomial dependence on $N$ and $\beta$ using $X^{f_2}_{\beta,c,1/\beta}$)\label{it:rapidmix}
		
		$$t^{f_2}_{mix}(\varepsilon) = \mathcal{O}\left(N^3 \left(\ln\left(\dfrac{2}{\varepsilon}\right) + \beta(c - \mcH_{*}) + \dfrac{\pi}{2} + N \ln 2\right)\right).$$
		
		In particular, if we take $\varepsilon = e^{-N}$ and $c$ to be chosen close enough to $\mcH_{*}$ such that
		$$\beta = \dfrac{1}{c - \mcH_*} \Theta(N),$$
		then
		$$t^{f_2}_{mix}(e^{-N}) = \mathcal{O}(N^4)$$
		while
		$$t^0_{mix}(e^{-N}) = \Omega((4e)^N N).$$
		
		\item(Torpid mean tunneling time to the global minimum with exponential dependence on $N$ using $X^0_{\beta}$)\label{it:torpidtunnel}
		For $\varepsilon < 1$, we have
		$$\sup_{\sigma \in \mathcal{LM}} \mathbb{E}_{\sigma}(\tau^0_{\sigma_*}) \geq \dfrac{e^{\beta m}}{N2^{N-1}} =  \Omega\left(4^N \right),$$
		where we recall $\sigma_* = \arg \min \mcH(\sigma).$
		
		\item(Rapid mean tunneling time to the global minimum with polynomial dependence on $N$ and $\beta$ using $X^{f_2}_{\beta,c,1/\beta}$)\label{it:rapidtunnel}
		For $\varepsilon < 0.4$, we have
		$$\sup_{\sigma \in \mathcal{LM}} \mathbb{E}_{\sigma}(\tau^{f_2}_{\sigma_*}) =  \mathcal{O}\left(N^3 \left(\ln 4 + \beta(c - \mcH_{*}) + \dfrac{\pi}{2} + N \ln 2\right)\right).$$
		In particular, if we choose $c$ to be close enough to $\mcH_{*}$ such that
		$$\beta = \dfrac{1}{c - \mcH_*} \Theta(N),$$
		then
		$$\sup_{\sigma \in \mathcal{LM}} \mathbb{E}_{\sigma}(\tau^{f_2}_{\sigma_*}) = \mathcal{O}(N^4).$$
		
		\item($X^0_{\beta}$ takes at least exponential in $\beta$ time to reach $\mcH_{*}$)\label{it:torpidreach}
		$$\mathcal{T}^{0}(\varepsilon) = \Omega\left(e^{\beta \delta}\right),$$
		where $\delta := \min_{i \in \llbracket N \rrbracket,~\sigma \in \mathcal{LM}} \mcH(\sigma^{(i)}) - \mcH(\sigma)$. In particular, if we consider small enough temperature such that $\beta \delta \geq N \ln 2 + \ln (4/\varepsilon)$, then
		$$\mathcal{T}^{0}(\varepsilon) = \Omega\left(\dfrac{2^N}{\varepsilon}\right).$$
		
		\item($X^{f_2}_{\beta,c,1/\beta}$ reaches $\mcH_{*}$ in polynomial time with high probability)\label{it:rapidreach}
		$$\mathcal{T}^{f_2}(\varepsilon) = \mathcal{O}\left(N^3 \left(\ln\left(\dfrac{2}{\varepsilon}\right) + \beta(c - \mcH_{*}) + \dfrac{\pi}{2} + N \ln 2\right)\right).$$
		In particular, if we take $\varepsilon = e^{-N}$ and $c$ to be chosen close enough to $\mcH_{*}$ such that
		$$\beta = \dfrac{1}{c - \mcH_*} \Theta(N),$$
		then
		$$\mathcal{T}^{f_2}(e^{-N}) = \mathcal{O}(N^4).$$
		
	\end{enumerate}
\end{theorem}

\begin{rk}\label{rk:mainrk}
	In the following remarks, we assume the same setting as in Theorem \ref{thm:main}. The aim of this remark is to show other transformations of $\mcH$ may not enjoy the same benefits as the proposed transformation $\mcH^{f}_{\alpha,c,1/\beta}$.
	\begin{enumerate}
		\item\label{it:rk1}(Slower or lack of concentration of the associated Gibbs distribution as $\beta \to \infty$) Suppose we consider transforming the target Hamiltonian function to
		$$\mcH \mapsto T \mcH$$
		or 
		$$\mcH \mapsto T \ln (\beta) \mcH.$$
		The benefit of these transformations is to remove the exponential dependency on $\beta$ in the relaxation time: according to \cite{HS88} the relaxation time scales at least in the order of $\Omega(e^{m})$ and $\Omega(e^{(\ln \beta)m})$ respectively. The disadvantages are twofold. First, if $m = \Omega(N)$, then the chain still mixes torpidly with exponential dependency on $N$. Second, the Gibbs distributions associated with these transformations concentrate at a slower rate on the set of global minima of $\mcH$ than $\pi^0_{\beta}$ as $\beta \to \infty$. 
		Writing $\mu_1$ and $\mu_2$ to be respectively the Gibbs distribution with Hamiltonian $T \mcH$ and $T (\ln \beta) \mcH$, then, for fixed $N$,
		\begin{align*}
			\mu_1(\Sigma_N \backslash \{\sigma_*\}) &\geq \dfrac{(2^N-1)e^{-(\mcH_{max} - \mcH_{*})}}{1+ (2^N-1) e^{-\Delta}} = \Omega(1), \\
			\mu_2(\Sigma_N \backslash \{\sigma_*\}) &\geq \dfrac{(2^N-1)e^{-(\ln \beta)(\mcH_{max} - \mcH_{*})}}{1+ (2^N-1) e^{-(\ln \beta)\Delta}} = \Omega(\beta^{\Delta - (\mcH_{max} - \mcH_{*})}), \\
			\pi^0_{\beta}(\Sigma_N \backslash \{\sigma_*\}) &\leq \mathcal{O}\left(e^{-\beta \Delta}\right),
		\end{align*}
		where the right hand side of the first inequality above does not depend on $\beta$, which does not converge to $0$ as $\beta \to \infty$. As such there are non-negligible biases if we approximately sample from $\mu_1$ or $\mu_2$ instead of $\pi^0_{\beta}$. On the other hand, the proposed target $\pi^f_{\alpha,c,1/\beta}$ can be arbitrarily close to $\pi^0_{\beta}$ provided that $\beta$ is large enough. 
		
		\item\label{it:rk2}(Exponential dependency on $\beta$) Suppose we consider changing the target Hamiltonian function to
		$$\mcH \mapsto \dfrac{\mcH}{\mcH_{max} - \mcH_{*}},$$
		or
		$$\mcH \mapsto \arctan(\mcH - \mcH_{*}).$$
		The advantage of these transformations is a reduced critical height: the new critical height associated with these transformations is denoted by $m^{\prime} \in [0,1]$, which is independent of the system size $N$. On the other hand, there are two possible drawbacks of considering these transformations. First, in practice it is perhaps difficult to estimate the oscillation of $\mcH$, i.e. $\mcH_{max} - \mcH_{*}$. Second, in the low-temperature regime as $\beta \to \infty$, according to \cite{HS88} the relaxation time scales at least in the order of
		$$\Omega(e^{\beta m^{\prime}}),$$
		which exhibits an exponential dependency on $\beta$. However, in the proposed landscape modification the relaxation time can be bounded above independently of $\beta$ as in Theorem \ref{thm:main} item \eqref{it:rapidrelax}.
		
		\item(Landscape modification enjoys the "best of both worlds") The proposed transformation $\mcH^{f_2}_{\beta,c,1/\beta}$ in Theorem \ref{thm:main} enjoys the advantages while addressing the drawbacks mentioned in both item \ref{it:rk1} and \ref{it:rk2} of Remark \ref{rk:mainrk}. Precisely, recall that
		\begin{align}\label{eq:mcHf2}
			\mcH^{f_2}_{\beta,c,1/\beta}(\sigma) = \beta (\mcH(\sigma) \wedge c - \mcH_{*}) + \arctan(\beta(\mcH(\sigma)-c)_+).
		\end{align}
		On the part of the landscape at or below the threshold parameter $c$, namely $\{\sigma;~\mcH(\sigma) \leq c\}$, $\mcH^{f_2}_{\beta,c,1/\beta}$ effectively exploits and concentrates on the original landscape, while on the part of the landscape above $c$, the difference in the modified Hamiltonian can be bounded independently in $\beta$ and the system size $N$, thus giving rise to rapid mixing and tunneling.
		
		\item(On other choices of $f$ such as $f_1(x) = x$)
		Suppose we instead consider
		$$\mcH^{f_1}_{1,c,1/\beta}(\sigma) = \beta (\mcH(\sigma) \wedge c - \mcH_{*}) + \ln(1+\beta(\mcH(\sigma)-c)_+).$$
		The advantage of $\mcH^{f_2}_{\beta,c,1/\beta}$ is that its oscillation above $c$ is $\mathcal{O}(1)$ owing to the $\arctan$ transformation, while for $\mcH^{f_1}_{1,c,1/\beta}$ its oscillation above $c$ is $\mathcal{O}(\ln \beta(\mcH_{max}-c))$ which can still be large depending on $\mcH$.
		
		\item(On tuning of the threshold parameter $c$) Given a Hamiltonian $\mcH$ which we know about its $\mcH_*$, we can set $c = \mcH_{*} + \delta$ where $\delta$ is small enough. In this way the landscape is appropriately modified with a modified critical height that is independent of $N$. For details we refer interested readers to Section \ref{sec:applications} where we consider a comprehensive suite of Ising model on graphs and the random energy model.
	\end{enumerate}
\end{rk}

In our second main result, we consider the setting of simulated annealing, that is, the temperature decreases according to a cooling schedule, instead of the fixed temperature case as in Theorem \ref{thm:main}. The crucial observation is similar to that of Theorem \ref{thm:main}: recall that the modified critical height can be bounded above by $$m^f_{\alpha,c,1/\beta} \leq \max_{\sigma \in \Sigma_N} \mcH^f_{\alpha,c,1/\beta}(\sigma) - \min_{\sigma \in \mathcal{LM}\cap\mathcal{GM}'}\mcH^f_{\alpha,c,1/\beta}(\sigma),$$
With our choice of $c \in (\mathcal{H}_*,\min_{\sigma \in \mathcal{LM}\cap\mathcal{GM}'}\mathcal{H}(\sigma)]$ the linear term in the difference of \eqref{eq:mcHf2} of $\beta$ gets canceled out, so that the relaxation time can be bounded above independently of the inverse temperature. This observation is crucial as this allows us to show that we can operate a power-law cooling schedule, namely $\beta_t = t^a$ with $a \in (0,1)$, unlike in the classical setting of simulated annealing where one adapts an appropriate logarithmic cooling schedule. 

With a slight abuse of notations, we write $\bm{\beta} = (\beta_t)_{t \geq 0}$ and $\mathbf{T} = (1/\beta_t)_{t \geq 0}$ to be the inverse temperature schedule and temperature schedule respectively. We also let $X^{f_2}_{\bm{\beta},c,\mathbf{T}}$ to be the annealing chain with transition semigroup defined to be, for any $\sigma, \eta \in \Sigma_N$,
$$\left(P^{f_2}_{\bm{\beta},c,\mathbf{T}}\right)^t(\sigma,\eta) := \exp\bigg\{\int_0^t L^{f_2}_{\beta_s,c,1/\beta_s}\,ds \bigg\}(\sigma,\eta),$$
where we recall the instantaneous infinitesimal generator $L^{f_2}_{\beta_s,c,1/\beta_s}$ at time $s \geq 0$ is first introduced in \eqref{eq:Llm}.

As a result of the power-law cooling schedule, the simulated annealing chain can approximately reach $\mcH_{*}$ in polynomial time in $N$ with high probability:

\begin{theorem}\label{thm:simanneal}
	Suppose that $\mcH$ has a unique global minimum, and we choose $f(x) = f_2(x) = x^2$, $\alpha_t = \beta_t = t^a$ with $a \in (0,1)$ and $c \in (\mathcal{H}_*,\min_{\sigma \in \mathcal{LM}\cap\mathcal{GM}'}\mathcal{H}(\sigma)]$. For any $\delta$ small enough such that $\delta < c - \mcH_{*}$, we then have, for $\varepsilon > 0$,
		$$\mP_{\sigma}\left(\mcH\left(X^{f_2}_{\bm{\beta},c,\mathbf{T}}(t)\right) \geq \mcH_{*} + \delta\right) \leq \varepsilon$$
	whenever 
	$$t \geq \max\{t_0,\tau_1,\tau_2,\tau_3\},$$
	where
	\begin{align*}
		t_0 = t_0(\mcH,N,a,\varepsilon) &:= \left(a (\mcH_{max}- \mcH_*) (1 + \varepsilon^2/3) \dfrac{3N^3 e^{\pi/2}}{4 \varepsilon^2}\right)^{\frac{1}{1-a}},\\
		\tau_1 = \tau_1(N,a,\varepsilon,\delta) &:= \left(\dfrac{1}{\delta}\left((N+1)\ln 2 + \ln \frac{1}{\varepsilon}\right)\right)^{1/a},\\
		\tau_2 = \tau_2(\mcH,N,a,\varepsilon) &:= \max\bigg\{\left(\dfrac{N^3 e^{\pi/2}}{2}\left(\mcH_{max}-\mcH_{*}\right)\right)^{\frac{1}{1-a}} ,\dfrac{N^3 e^{\pi/2}}{2} \ln \left(\dfrac{3(2^N+1)}{\varepsilon^2}\right)\bigg\}, \\
		\tau_3 =\tau_3(\mcH,N,a,\varepsilon) &:= \dfrac{N^3 e^{\pi/2}}{2}\ln \left(\dfrac{3}{\varepsilon^2} e^{4 t_0} t_0^a (\mcH_{max} - \mcH_{*})\right).
	\end{align*}
	In particular, if the oscillation $\mcH_{max} - \mcH_{*}$ is at most polynomial in $N$ and $\varepsilon = \delta = \mathcal{O}(\frac{1}{N})$ then the simulated annealing chain takes at most polynomial time in $N$ to reach an approximate global minimum with high probability.	
\end{theorem}

\begin{rk}
	In the following remarks, we assume the same setting as in Theorem \ref{thm:simanneal}.
	\begin{enumerate}
		\item The assumption that $\mcH_{max} - \mcH_{*}$ is at most polynomial in $N$ can be readily verified in all models considered in Section \ref{sec:applications}.
		
		\item(Improving the convergence rate of the instantaneous Gibbs distribution from $\mathcal{O}(t^{-D_1})$ to $\mathcal{O}(e^{- D_2 t^a})$ by landscape modification)
		
		In classical simulated annealing \cite{Miclo92AIHP} or kinetic simulated annealing \cite{M18} on $\mathbb{R}^d$, the instantaneous classical Gibbs distribution $\pi^0_{\beta_t}\left(\{\sigma;~\mcH(\sigma) \geq \mcH_{*} + \delta\}\right)$ at time $t$ converges at a rate of $\mathcal{O}(t^{-D_1})$ with $D_1 = D_1(m,\delta) \in (0,1)$ and $\delta$ appears as in Theorem \ref{thm:simanneal} under a logarithmic cooling schedule of the form $\beta_t = \frac{\ln t}{E}$ with $E > m$. On the other hand the landscape-modified Gibbs distribution at time $t$, see e.g. the right hand side of \eqref{eq:anneal1} below, is shown to converge at a rate of $\mathcal{O}(e^{- D_2 t^a})$ with $D_2 = D_2(\delta)$ owing to the power-law cooling $\beta_t = t^a$.
		
		Second, it is known in the literature that operating a cooling schedule faster than the logarithmic schedule in classical simulated annealing \cite{ChoiMPRF} or kinetic annealing \cite{M18} may end up the process being stuck in a local minimum or the basin of attraction of a local minimum. On the other hand, one can safely operate a power-law cooling, which is faster than logarithmic cooling, on the modified landscape with a long-time convergence guarantee thanks to Theorem \ref{thm:simanneal}.
		
	\end{enumerate}

\end{rk}

\section{Proof of the main results}\label{sec:pfmain}

\subsection{Refined spectral gap lower bound for $X^f_{\alpha,c,1/\beta}$}
We first prove a lower bound on $\lambda_2(-L^f_{\alpha,c,1/\beta})$. While the proof strategy is exactly the same as \cite{HS88}, we note that we choose a particular set of paths, namely $(\Lambda^{\eta,\sigma})_{\eta,\sigma}$, instead of $(\Gamma^{\eta,\sigma})_{\eta,\sigma}$. Recall that the definition of these paths are introduced in Section \ref{sec:spinlandmod}. The advantage of choosing the paths $(\Lambda^{\eta,\sigma})_{\eta,\sigma}$ is that it ensures the prefactor in front of the exponential term in the lower bound of the spectral gap estimate is inverse polynomial in $N$. More precisely, as we can see in Lemma \ref{lem:sglowerbd} below, the prefactor is of the order $N^{-3}$, and consequently this yields rapid mixing of $X^{f_2}_{\beta,c,1/\beta}$.

\begin{lemma}\label{lem:sglowerbd}
	For any $\beta > 0$, and suppose that $f,c$ are chosen to satisfy Assumption \ref{assump:fc}, then we have 
	\begin{align*}
		\lambda_2(-L^f_{\alpha,c,1/\beta}) \geq \dfrac{2}{N^3} e^{-m^f_{\alpha,c,1/\beta}}.
	\end{align*}
	where we recall that $m^f_{\alpha,c,1/\beta}$ is introduced in \eqref{eq:rch}. In particular, if $\mcH$ has a unique global minimum, then by choosing $f(x) = f_2(x) = x^2$, $\alpha = \beta$ and $c \in (\mathcal{H}_*,\min_{\sigma \in \mathcal{LM}\cap\mathcal{GM}'}\mathcal{H}(\sigma)]$,
	we have
	$$\lambda_2(-L^{f_2}_{\beta,c,1/\beta}) \geq \dfrac{2}{N^3} e^{-\pi/2}.$$
\end{lemma}
\begin{proof}
	For any $i \neq \eta$, we pick a path $\lambda^{i,\eta} = (\lambda^{i,\eta}_k)_{k=1}^{n(i,\eta)}$ from $i$ to $\eta$ by flipping the coordinates that differ one at a time and from left to right. The advantage of choosing these paths is that $\max_{i \neq \eta} n(i,\eta) \leq N$. For $z, w \in \Sigma_N$, we denote the indicator function $\chi_{z,w}$ to be
	\begin{align*}
		\chi_{z,w}(i,\eta) = \begin{cases}
			1, \quad \text{for some}~ 1 \leq k < n(i,\eta), ~ \lambda^{i,\eta}_k = z, \lambda^{i,\eta}_{k+1} = w, \\
			0, \quad \text{otherwise.}
		\end{cases}
	\end{align*}
	Let $\alpha(z,w) := \mu^{SRW}(z) P^{SRW}(z,w)$. If $\alpha(z,w) = 0$, then $\chi_{z,w}(i,\eta) = 0$ for all $i$. In the following we take $\chi_{z,w}(i,\eta)/\alpha(z,w) = 0$ if $\chi_{z,w}(i,\eta) = 0$. For any $g$ such that $\pi^f_{\alpha,c,1/\beta}(g) = 0$, we compute that
	\begin{align*}
		\langle g,g \rangle_{\pi^f_{\alpha,c,1/\beta}} &= \sum_{i} g^2(i) \pi^f_{\alpha,c,1/\beta}(i) \\
		&\leq \sum_{\eta} \sum_i (g(i) - g(\eta))^2 \pi^f_{\alpha,c,1/\beta}(i) \mu^{SRW}(\eta) \\
		&= \sum_{\eta}  \sum_{i} \left(\sum_{k=1}^{n(i,\eta)-1} g(\lambda_k^{i,\eta}) - g(\lambda_{k+1}^{i,\eta})\right)^2 \pi^f_{\alpha,c,1/\beta}(i) \mu^{SRW}(\eta) \\
		&\leq N \sum_{\eta} \sum_{i } \sum_{k=1}^{n(i,\eta)-1} (g(\lambda_k^{i,\eta}) - g(\lambda_{k+1}^{i,\eta}))^2 \pi^f_{\alpha,c,1/\beta}(i)\mu^{SRW}(\eta) \\
		&= N \sum_{\eta} \sum_{i } \sum_{z} \sum_w \chi_{z,w}(i,\eta) \dfrac{\pi^f_{\alpha,c,1/\beta}(i) Z^f_{\alpha,c,1/\beta} (1/2^N) \mu^{SRW}(\eta)}{\alpha(z,w) e^{-\mcH^f_{\alpha,c,1/\beta}(z) \vee \mcH^f_{\alpha,c,1/\beta}(w)}} \\
		&\quad (g(z) - g(w))^2  \dfrac{\alpha(z,w) e^{-\mcH^f_{\alpha,c,1/\beta}(z) \vee \mcH^f_{\alpha,c,1/\beta}(w)}}{Z^f_{\alpha,c,1/\beta} (1/2^N)} \\
		&\leq N \left(\max_{z,w \in \Sigma_N} \sum_{\eta}\sum_{i } \chi_{z,w}(i,\eta) \dfrac{\pi^f_{\alpha,c,1/\beta}(i) Z^f_{\alpha,c,1/\beta} (1/2^N) \mu^{SRW}(\eta)}{\alpha(z,w) e^{-\mcH^f_{\alpha,c,1/\beta}(z) \vee \mcH^f_{\alpha,c,1/\beta}(w)}}\right) \langle - L^f_{\alpha,c,1/\beta} g,g \rangle_{\pi^f_{\alpha,c,1/\beta}} \\
		&\leq \left(N \max_{z,w \in \Sigma_N} \sum_{\eta}\sum_{i } \dfrac{\chi_{z,w}(i,\eta) \mu^{SRW}(i) \mu^{SRW}(\eta)}{\alpha(z,w)}\right) e^{m^f_{\alpha,c,1/\beta}} \langle -L^f_{\alpha,c,1/\beta} g,g \rangle_{\pi^f_{\alpha,c,1/\beta}}.
	\end{align*}
	To finish the proof, we note that, according to the conductance bound for the hypercube $\mathbb{Z}^N_2$ as in \cite[Example $2.2$]{DS91}, we have
	\begin{align*}
		\max_{z,w \in \Sigma_N} \sum_{\eta}\sum_{i } \dfrac{\chi_{z,w}(i,\eta) \mu^{SRW}(i) \mu^{SRW}(\eta)}{\alpha(z,w)} \leq \dfrac{N^2}{2}.
	\end{align*}
	In particular, when $f = f_2$, we recall from \eqref{eq:mcHf2alphabeta} that
	$$\mcH^{f_2}_{\alpha,c,1/\beta} (\sigma) = \beta (c\wedge \mcH(\sigma)-\mcH_{*}) + \sqrt{\dfrac{\beta}{\alpha}} \arctan(\sqrt{\alpha \beta}(\mcH(\sigma)-c)_+),$$
	and since 
	$$m^f_{\alpha,c,1/\beta} \leq \max_{\sigma \in \Sigma_N} \mcH^f_{\alpha,c,1/\beta}(\sigma) - \min_{\sigma \in \mathcal{LM}\cap\mathcal{GM}'}\mcH^f_{\alpha,c,1/\beta}(\sigma)$$
	the given choices of $c$ and $\alpha = \beta$ yields
	$$m^{f_2}_{\beta,c,1/\beta} \leq \arctan \left(\beta\left(\max_{\sigma \in \Sigma_N}\mcH(\sigma)-c\right)\right) \leq \pi/2.$$
\end{proof}

\subsection{Proof of Theorem \ref{thm:main} item \eqref{it:torpidrelax}}

It is classical that, for some constants $C_1(N), C_2(N)$ that only depend on $N$, we have, according to \cite[Theorem $2.1$]{HS88},
\begin{align*}
	C_1(N) e^{-\beta m} \leq \lambda_2(-L^0_{\beta}) \leq C_2(N) e^{-\beta m},
\end{align*}
where $C_2(N) = 4^N$. The desired result follows.

\subsection{Proof of Theorem \ref{thm:main} item \eqref{it:rapidrelax}}

The desired result follows by directly applying the refined spectral gap lower bound \ref{lem:sglowerbd}.

\subsection{Proof of Theorem \ref{thm:main} item \eqref{it:torpidmix}}

First, we recall from \cite[Theorem $12.5$]{LPW17} or \cite[Theorem $8(a)$]{Aldous82} that
$$t_{mix}^0(\varepsilon) \geq t^0_{rel} \log \left(\dfrac{1}{2\varepsilon}\right).$$
The desired result follows from combining the above inequality with item \eqref{it:torpidrelax}.

\subsection{Proof of Theorem \ref{thm:main} item \eqref{it:rapidmix}}

First, using the triangle inequality gives
\begin{align}\label{eq:triangle}
	\norm{(P^{f_2}_{\beta,c,1/\beta})^t(\sigma,\cdot)-\pi^0_{\beta}}_{TV} \leq \norm{(P^{f_2}_{\beta,c,1/\beta})^t(\sigma,\cdot)-\pi^{f_2}_{\beta,c,1/\beta}}_{TV} + \norm{\pi^{f_2}_{\beta,c,1/\beta} - \pi^0_{\beta}}_{TV}.
\end{align}

Note that by our choice of $\beta$ and the assumption that $\mcH$ has a unique global minimum we have
\begin{align}\label{eq:pif2pi0}
	\norm{\pi^{f_2}_{\beta,c,1/\beta} - \pi^0_{\beta}}_{TV} &= \dfrac{1}{2}\left|\dfrac{Z^0_{\beta} - Z^{f_2}_{\beta,c,1/\beta}}{Z^0_{\beta} Z^{f_2}_{\beta,c,1/\beta}}\right| + \dfrac{1}{2} \sum_{\sigma \notin \mathcal{GM}} |\pi^{f_2}_{\beta,c,1/\beta}(\sigma) - \pi^0_{\beta}(\sigma)| \nonumber\\
	&\leq \sum_{\sigma \notin \mathcal{GM}} e^{-\beta (\mcH(\sigma) - \mcH_{*})} + \sum_{\sigma \notin \mathcal{GM}} e^{-\mcH^{f_2}_{\beta,c,1/\beta}(\sigma)} \nonumber\\
	&\leq (2^N-1) e^{-\beta \Delta} + (2^N-1) e^{-\beta ((c - \mcH_{*})\wedge \Delta)} \nonumber\\
	&\leq \varepsilon/4 + \varepsilon/4 = \varepsilon/2.
\end{align}
In other words, at low enough temperature $\pi^{f_2}_{\beta,c,1/\beta}$ is a reasonably accurate proxy or substitute of $\pi$ since their total variation distance is less than or equal to $\varepsilon/2$.

Next, we compute that
$$\ln \left(\dfrac{1}{\min_{\sigma \in \Sigma_N} \pi^{f_2}_{\beta,c,1/\beta}(\sigma)}\right) \leq \max_{\sigma} \mcH^{f_2}_{\beta,c,1/\beta}(\sigma) + \ln Z^{f_2}_{\beta,c,1/\beta} \leq \beta(c - \mcH_{*}) + \dfrac{\pi}{2} + N \ln 2,$$
where the last inequality follows from the parameter choice of $\alpha = \beta$ and $f_2(x) = x^2$. Using the classical upper bound in \cite[Theorem $20.6$]{LPW17} yields, for 
$$t \geq \mathcal{O}(N^3) \left(\ln\left(\dfrac{2}{\varepsilon}\right) + \beta(c - \mcH_{*}) + \dfrac{\pi}{2} + N \ln 2\right) \geq t^{f_2}_{rel} \ln \left(\dfrac{2}{\varepsilon \min_{\sigma \in \Sigma_N} \pi^{f_2}_{\beta,c,1/\beta}(\sigma)}\right),$$
then for any $\sigma \in \Sigma_N$,
\begin{align}\label{eq:convergetopif2}
	\norm{(P^{f_2}_{\beta,c,1/\beta})^t(\sigma,\cdot)-\pi^{f_2}_{\beta,c,1/\beta}}_{TV} \leq \varepsilon/2.
\end{align}
Collecting \eqref{eq:triangle}, \eqref{eq:pif2pi0} and \eqref{eq:convergetopif2} gives the desired result.

\subsection{Proof of Theorem \ref{thm:main} item \eqref{it:torpidtunnel}}

First, we recall the potential-theoretic approach in the study of mean hitting time as in \cite{BH15}. The capacity and equilibrium potential of the pair $(A,B)$, where $A,B \subseteq \Sigma_N$, are defined to be as in \cite[Chapter $7.2$]{BH15}:
\begin{align}
h_{A,B}^{f}(x) &:= \mP_x(\tau_{A}^f < \tau_{B}^f), \nonumber \\
\mathrm{cap}^{f}_{\alpha,c,1/\beta}(A,B) &:= \inf_{g: g|_{A} = 1, g|_{B} = 0} \langle - L^{f}_{\alpha,c,1/\beta}g,g \rangle_{\pi^f_{\alpha,c,1/\beta}} = \langle -L^{f}_{\alpha,c,1/\beta} h^f_{A,B}, h^f_{A,B} \rangle_{\pi^f_{\alpha,c,1/\beta}}. \label{eq:capvariation}
\end{align}
If $A = \{x\}$ and/or $B = \{y\}$ are singletons, we write $h_{x,y}^{f} = h_{\{x\},\{y\}}^{f}$ and $\mathrm{cap}^f_{\alpha,c,1/\beta}(x,y) = \mathrm{cap}^f_{\alpha,c,1/\beta}(\{x\},\{y\})$. An useful formula connecting the mean hitting time with capacity and equilibrium potential is
\begin{align}\label{eq:potentialhit}
\E_x(\tau^f_B) = \dfrac{1}{\mathrm{cap}^f_{\alpha,c,1/\beta}(x, B)} \sum_{y \in \Sigma_N} \pi^f_{\alpha,c,1/\beta}(y) h_{x, B}^{f}(y).
\end{align}

We now proceed to prove item \eqref{it:torpidtunnel}. First, we pick $x \in \mathcal{LM}$ that attains $m$, i.e. $m = \mcH(y) - \mcH(x)$ for some $y \in \Sigma_N$. Using \cite[Lemma $16.11$]{BH15} we have
$$e^{\beta \mcH(y)} Z^0_{\beta} \mathrm{cap}^0_{\beta}(x,\sigma_*) \leq C_2 \leq N 2^{N-1},$$
where $C_2$ is bounded by the maximum number of edges in the graph in which in our case of the hypercube is $N2^{N-1}$ according to \cite{DS91}. This, combined with \eqref{eq:potentialhit}, yields
\begin{align*}
	\sup_{\sigma \in \mathcal{LM}} \mathbb{E}_{\sigma}(\tau^0_{\sigma_*}) \geq \mathbb{E}_{x}(\tau^0_{\sigma_*}) 
	&\geq \dfrac{1}{\mathrm{cap}^0_{\beta}(x,\sigma_*)} \pi^0_{\beta}(x)\\ 
	&\geq \dfrac{e^{\beta m}}{N2^{N-1}} \\
	&\geq \dfrac{16^N}{N 2^{N-1}} =	\Omega\left(4^N \right) .
\end{align*}

\subsection{Proof of Theorem \ref{thm:main} item \eqref{it:rapidtunnel}}

First, we define
$$\mathbf{t}_{mix}^{f_2} := \inf\bigg\{t \geq 0;~ \sup_{\sigma \in \Sigma_N} \norm{(P^{f_2}_{\beta,c,1/\beta})^t(\sigma,\cdot)-\pi^{f_2}_{\beta,c,1/\beta}}_{TV} \leq 1/4\bigg\}.$$
It follows from the proof of Theorem \ref{thm:main} item \eqref{it:rapidmix} that
$$\mathbf{t}_{mix}^{f_2} = \mathcal{O}\left(N^3 \left(\ln 4 + \beta(c - \mcH_{*}) + \dfrac{\pi}{2} + N \ln 2\right)\right).$$

Using the result that the mean hitting time of large sets is equivalent up to universal and chain-independent constant to the total variation mixing time \cite{Oliveria12,PS15}, in this case we have, for any $\sigma \in \Sigma_N$,
\begin{align*}
	\mathbb{E}_{\sigma}(\tau^{f_2}_{\sigma_*}) &\leq \sup_{\sigma, A;~\pi^{f_2}_{\beta,c,1/\beta}(A) \geq 0.4 }\mathbb{E}_{\sigma}(\tau^{f_2}_{A}) \\
	&= \mathcal{O}(\mathbf{t}_{mix}^{f_2}),
\end{align*}
where the first inequality follows from the choice that $\varepsilon < 0.4$ such that
$\pi^{f_2}_{\beta,c,1/\beta}(\sigma_*) \geq 1 - \varepsilon > 0.6$ while $\pi^{f_2}_{\beta,c,1/\beta}(\Sigma_N \backslash \{\sigma_*\}) \leq \varepsilon < 0.4$.

\subsection{Proof of Theorem \ref{thm:main} item \eqref{it:torpidreach}}

First, for $\sigma \in \mathcal{LM}$ we calculate that 
$$\sum_{\sigma_3: \sigma_3 \neq \sigma} L^{f_2}_{\beta,c,1/\beta}(\sigma,\sigma_3) \leq e^{-\beta \delta},$$
and hence, for any $t \geq 0$,
\begin{align*}
	\mP_{\sigma}(X^0_{\beta}(t) = \sigma) \geq \exp\bigg\{-\sum_{\sigma_3: \sigma_3 \neq \sigma} L^{f_2}_{\beta,c,1/\beta}(\sigma,\sigma_3) t\bigg\} \geq e^{-e^{-\beta \delta}t}.
\end{align*}
In other words, with probability at least $e^{-1}$ the chain is still stuck at the given local minimum $\sigma$ for $t \leq e^{\beta \delta}$. The desired result follows.

\subsection{Proof of Theorem \ref{thm:main} item \eqref{it:rapidreach}}

First, we recall that with our choice of $\beta$ this leads to
$$\pi^{f_2}_{\beta,c,1/\beta}(\Sigma_N \backslash \{\sigma_*\}) \leq (2^N-1) e^{-\beta \Delta} \leq \varepsilon/4.$$

Next, for $$t \geq \mathcal{O}(N^3) \left(\ln\left(\dfrac{2}{\varepsilon}\right) + \beta(c - \mcH_{*}) + \dfrac{\pi}{2} + N \ln 2\right),$$
we note that as in the proof of Theorem \ref{thm:main} item \eqref{it:rapidmix}, for any $\sigma \in \Sigma_N$,
\begin{align*}
\norm{(P^{f_2}_{\beta,c,1/\beta})^t(\sigma,\cdot)-\pi^{f_2}_{\beta,c,1/\beta}}_{TV} \leq \varepsilon/2.
\end{align*}

Collecting the above gives
\begin{align*}
	\mP_{\sigma}(\mcH(X^{f_2}_{\beta,c,1/\beta}(t)) > \mcH_*) \leq \pi^{f_2}_{\beta,c,1/\beta}(\Sigma_N \backslash \{\sigma_*\}) + \norm{(P^{f_2}_{\beta,c,1/\beta})^t(\sigma,\cdot)-\pi^{f_2}_{\beta,c,1/\beta}}_{TV} \leq \varepsilon,
\end{align*}
and hence $\mathcal{T}^{f_2}(\varepsilon) \leq \mathcal{O}(N^3) \left(\ln\left(\dfrac{2}{\varepsilon}\right) + \beta(c - \mcH_{*}) + \dfrac{\pi}{2} + N \ln 2\right)$.

\subsection{Proof of Theorem \ref{thm:simanneal}}

First, we present a lemma that computes the time derivative of $\pi^{f_2}_{\beta_t,c,1/\beta_t}$:

\begin{lemma}\label{lem:timedpif2}
	For any $\sigma \in \Sigma_N$,
	$$\left| \dfrac{\partial}{\partial t} \pi^{f_2}_{\beta_t,c,1/\beta_t}(\sigma) \right| \leq \gamma_t \pi^{f_2}_{\beta_t,c,1/\beta_t}(\sigma),$$
	where
	$$\gamma_t := \dfrac{\partial \beta_t}{\partial t} \left(\mcH_{max} - \mcH_{*}\right)$$
	and we recall that $\mcH_{max} = \max_{\sigma \in \Sigma_N} \mcH(\sigma)$. In particular, if $\beta_t = t^a$ with $a \in (0,1)$, then
	$$\gamma_t = a t^{a-1}\left(\mcH_{max} - \mcH_{*}\right).$$
\end{lemma}

\begin{proof}
	Recall that with our choice of $f_2$, we have
	\begin{align*}
		\mcH^{f_2}_{\beta_t,c,1/\beta_t}(\sigma) &= \beta_t (\mcH(\sigma) \wedge c - \mcH_{*}) + \arctan(\beta_t(\mcH(\sigma)-c)_+) \\
		\dfrac{\partial}{\partial t} \mcH^{f_2}_{\beta_t,c,1/\beta_t}(\sigma) &= \left(\dfrac{\partial}{\partial t} \beta_t\right) (\mcH(\sigma) \wedge c - \mcH_{*}) + \dfrac{1}{1+ \beta_t^2 (\mcH(\sigma)-c)_+^2}\left(\dfrac{\partial}{\partial t} \beta_t\right) (\mcH(\sigma)-c)_+ \\
		\dfrac{\partial}{\partial t} \pi^{f_2}_{\beta_t,c,1/\beta_t}(\sigma) &= \left(\left(- \dfrac{\partial}{\partial t} \mcH^{f_2}_{\beta_t,c,1/\beta_t}(\sigma)\right) - \pi^{f_2}_{\beta_t,c,1/\beta_t}\left(- \dfrac{\partial}{\partial t} \mcH^{f_2}_{\beta_t,c,1/\beta_t}\right) \right)\pi^{f_2}_{\beta_t,c,1/\beta_t}(\sigma).
	\end{align*}
	The desired result follows if we take absolute value on both sides of the last equality above and use that
	$$\left(\left(- \dfrac{\partial}{\partial t} \mcH^{f_2}_{\beta_t,c,1/\beta_t}(\sigma)\right) - \pi^{f_2}_{\beta_t,c,1/\beta_t}\left(- \dfrac{\partial}{\partial t} \mcH^{f_2}_{\beta_t,c,1/\beta_t}\right) \right) \leq \gamma_t.$$
\end{proof}

Now, we return to the proof of Theorem \ref{thm:simanneal}. Recall that we are considering the power-law cooling schedule with $\beta_t = t^a$ for $a \in (0,1)$ and $t \geq 0$. Using the definition of total variation distance, we have

\begin{align}
	\mP_{\sigma}\left(\mcH\left(X^{f_2}_{\bm{\beta},c,\mathbf{T}}(t)\right) \geq \mcH_{*} + \delta\right) &\leq \pi^{f_2}_{\beta_t,c,1/\beta_t}\left(\{\sigma \in \Sigma_N;~\mcH(\sigma) \geq \mcH_{*} + \delta\}\right) \label{eq:anneal1}\\
	&\quad + \norm{(P^{f_2}_{\bm{\beta},c,\mathbf{T}})^t(\sigma,\cdot)-\pi^{f_2}_{\beta_t,c,1/\beta_t}}_{TV}. \label{eq:anneal2}
\end{align}

We proceed to bound the right hand side in both \eqref{eq:anneal1} and \eqref{eq:anneal2}. For \eqref{eq:anneal1}, we have

\begin{align*}
	\pi^{f_2}_{\beta_t,c,1/\beta_t}\left(\{\sigma \in \Sigma_N;~\mcH(\sigma) \geq \mcH_{*} + \delta\}\right) &= \sum_{\sigma;~\mcH(\sigma) \geq \mcH_{*} + \delta} \dfrac{e^{-\mcH^{f_2}_{\beta_t,c,1/\beta_t}(\sigma)}}{Z^{f_2}_{\beta_t,c,1/\beta_t}} \\
	&\leq \sum_{\sigma;~\mcH(\sigma) \geq \mcH_{*} + \delta} e^{-\beta_t\left(\mcH(\sigma) \wedge c - \mcH_{*}\right)} \\
	&\leq 2^N e^{-t^a \delta} \overset{\mathrm{set}}{=} \dfrac{\varepsilon}{2},
\end{align*}
where we utilize $Z^{f_2}_{\beta_t,c,1/\beta_t} \geq 1$ and $\delta < c - \mcH_{*}$ in the first and the second inequality respectively. Thus, if 
\begin{align}\label{eq:t0}
	t \geq \left(\dfrac{1}{\delta}\left((N+1)\ln 2 + \ln \frac{1}{\varepsilon}\right)\right)^{1/a} =: \tau_1(N,a,\varepsilon,\delta) = \tau_1,
\end{align}
then the right hand side in \eqref{eq:anneal1} is less than or equal to $\varepsilon/2$. For \eqref{eq:anneal2}, first we define
\begin{align*}
	\left(Q^{f_2}_{\bm{\beta},c,\mathbf{T}}\right)^t(\sigma,\cdot) := \dfrac{\left(P^{f_2}_{\bm{\beta},c,\mathbf{T}}\right)^t(\sigma,\cdot)}{\pi^{f_2}_{\beta_t,c,1/\beta_t}(\sigma)}.
\end{align*}
Using the Cauchy-Schwartz inequality leads to
\begin{align}
	\norm{(P^{f_2}_{\bm{\beta},c,\mathbf{T}})^t(\sigma,\cdot)-\pi^{f_2}_{\beta_t,c,1/\beta_t}}_{TV} &\leq \dfrac{1}{2} \left(\sum_{\eta \in \Sigma_N} \pi^{f_2}_{\beta_t,c,1/\beta_t}(\eta) \left(\left(Q^{f_2}_{\bm{\beta},c,\mathbf{T}}\right)^t(\sigma,\cdot) - 1\right)^2\right)^{1/2} \nonumber \\
	&=:\dfrac{1}{2} \left(F^{f_2}_{\bm{\beta},c,\mathbf{T}}(t)\right)^{1/2}. \label{eq:tvbound}
\end{align}
According to \cite[equation $(2.13)$]{Gidas85}, the function $F^{f_2}_{\bm{\beta},c,\mathbf{T}}(t)$ satisfies the following differential inequality:
$$\dfrac{d}{d t} F^{f_2}_{\bm{\beta},c,\mathbf{T}}(t) \leq \left(-2 \lambda_2(-L^{f_2}_{\beta_t,c,1/\beta_t}) + \gamma_t\right)F^{f_2}_{\bm{\beta},c,\mathbf{T}}(t) + \gamma_t.$$

Now, with our choice of $c$, $\beta_t = t^a$ and the temperature-independent lower bound of the spectral gap in Lemma \ref{lem:sglowerbd}, we have as $t \to \infty$
$$\dfrac{\gamma_t}{\lambda_2(-L^{f_2}_{\beta_t,c,1/\beta_t})} \to 0,$$
and hence
$$\dfrac{\gamma_t}{2\lambda_2(-L^{f_2}_{\beta_t,c,1/\beta_t}) - \gamma_t} \to 0.$$

As a result, we can now invoke \cite[Lemma $6$]{Miclo92AIHP} to see that
\begin{align}
	F^{f_2}_{\bm{\beta},c,\mathbf{T}}(t) &\leq F^{f_2}_{\bm{\beta},c,\mathbf{T}}(0) \exp\bigg\{ - \int_0^t 2 \lambda_2(-L^{f_2}_{\beta_s,c,1/\beta_s}) - \gamma_s \,ds \bigg\} \label{eq:first} \\
	&\quad + \exp\bigg\{ - \int_0^t 2 \lambda_2(-L^{f_2}_{\beta_s,c,1/\beta_s}) - \gamma_s \,ds \bigg\} \int_0^{t_0} \gamma_s \exp\bigg\{  \int_0^s 2 \lambda_2(-L^{f_2}_{\beta_u,c,1/\beta_u}) - \gamma_u \,du \bigg\}\, ds \label{eq:middle}\\
	&\quad +  \exp\bigg\{ - \int_0^t 2 \lambda_2(-L^{f_2}_{\beta_s,c,1/\beta_s}) - \gamma_s \,ds \bigg\} \int_{t_0}^{t} \gamma_s \exp\bigg\{  \int_0^s 2 \lambda_2(-L^{f_2}_{\beta_u,c,1/\beta_u}) - \gamma_u \,du \bigg\}\, ds \label{eq:larget0}
\end{align}
for some $t_0 \geq 0$. We proceed to show that if we take $t_0$ large enough then \eqref{eq:larget0} is less than or equal to $\varepsilon^2/3$. Using Lemma \ref{lem:sglowerbd} and \ref{lem:timedpif2} we have
\begin{align*}
	\dfrac{\gamma_t}{2\lambda_2(-L^{f_2}_{\bm{\beta},c,\mathbf{T}}) - \gamma_t} \leq \dfrac{a t^{a-1}(\mcH_{max} - \mcH_{*})}{\frac{4}{N^3} e^{-\pi/2} - a t^{a-1}(\mcH_{max} - \mcH_{*})} \overset{\mathrm{set}}{\leq} \dfrac{\varepsilon^2}{3}.
\end{align*}
Thus, if we take
\begin{align}\label{eq:t_0}
	t_0 = t_0(\mcH,N,a,\varepsilon):= \left(a (\mcH_{max}- \mcH_*) (1 + \varepsilon^2/3) \dfrac{3N^3 e^{\pi/2}}{4 \varepsilon^2}\right)^{\frac{1}{1-a}},
\end{align}
such that when $t \geq t_0$, then using \cite[proof of Lemma $6$]{Miclo92AIHP}
$$\exp\bigg\{ - \int_0^t 2 \lambda_2(-L^{f_2}_{\beta_s,c,1/\beta_s}) - \gamma_s \,ds \bigg\} \int_{t_0}^{t} \gamma_s \exp\bigg\{  \int_0^s 2 \lambda_2(-L^{f_2}_{\beta_u,c,1/\beta_u}) - \gamma_u \,du \bigg\}\, ds \leq \dfrac{\varepsilon^2}{3}.$$

We proceed to show \eqref{eq:first} can be bounded by $\varepsilon^2/3$ for large enough $t$.  Using Lemma \ref{lem:sglowerbd} and \ref{lem:timedpif2} again we consider
\begin{align}\label{eq:e2lambda2}
	\exp\bigg\{ - \int_0^t 2 \lambda_2(-L^{f_2}_{\beta_s,c,1/\beta_s}) - \gamma_s \,ds \bigg\} \leq \exp\bigg\{- \dfrac{4}{N^3} e^{-\pi/2}t + t^a (\mcH_{max}-\mcH_{*})\bigg\} \leq \exp \bigg\{ - \dfrac{2}{N^3} e^{-\pi/2}t\bigg\},
\end{align}
where the last inequality holds whenever
\begin{align}\label{eq:tau21}
	t > \left(\dfrac{N^3 e^{\pi/2}}{2}\left(\mcH_{max}-\mcH_{*}\right)\right)^{\frac{1}{1-a}}.
\end{align}
Next, we calculate that
\begin{align}\label{eq:Ff20}
	F^{f_2}_{\beta_0,c,1/\beta_0}(0) = \pi^{f_2}_{\beta_0,c,1/\beta_0}(\sigma) \left(\dfrac{1}{\pi^{f_2}_{\beta_0,c,1/\beta_0}(\sigma)} -1\right)^2 + \sum_{\eta \neq \sigma} \pi^{f_2}_{\beta_0,c,1/\beta_0}(\eta) \leq \dfrac{1}{\min_{\sigma \in \Sigma_N} \pi^{f_2}_{\beta_0,c,1/\beta_0}(\sigma)}+1 \leq 2^N+1.
\end{align}
Collecting both \eqref{eq:e2lambda2} and \eqref{eq:Ff20} and substituting back to \eqref{eq:first} we see that
\begin{align*}
	 F^{f_2}_{\beta_0,c,1/\beta_0}(0) \exp\bigg\{ - \int_0^t 2 \lambda_2(-L^{f_2}_{\beta_s,c,1/\beta_s}) - \gamma_s \,ds \bigg\} \leq (2^N+1)\exp \bigg\{ - \dfrac{2}{N^3} e^{-\pi/2}t\bigg\} \leq \dfrac{\varepsilon^2}{3},
\end{align*}
whenever
\begin{align}\label{eq:tau22}
	t \geq \dfrac{N^3 e^{\pi/2}}{2} \ln \left(\dfrac{3(2^N+1)}{\varepsilon^2}\right).
\end{align}
In summary, collecting \eqref{eq:tau21} and \eqref{eq:tau22} we have whenever
\begin{align}
	t \geq \tau_2 = \tau_2(\mcH,N,a,\varepsilon) := \max\bigg\{\left(\dfrac{N^3 e^{\pi/2}}{2}\left(\mcH_{max}-\mcH_{*}\right)\right)^{\frac{1}{1-a}} ,\dfrac{N^3 e^{\pi/2}}{2} \ln \left(\dfrac{3(2^N+1)}{\varepsilon^2}\right)\bigg\},
\end{align}
then the right hand side of \eqref{eq:first} is bounded above by $\varepsilon^2/3$.

Finally, we handle \eqref{eq:middle}. We see that
\begin{align}\label{eq:middlesecond}
	\int_0^{t_0} \gamma_s \exp\bigg\{  \int_0^s 2 \lambda_2(-L^{f_2}_{\beta_u,c,1/\beta_u}) - \gamma_u \,du \bigg\}\, ds \leq \int_0^{t_0} \gamma_s e^{4s}\, ds \leq e^{4 t_0} t_0^a (\mcH_{max} - \mcH_{*}),
\end{align}
where we make use of the fact that $\lambda_2(-L^{f_2}_{\beta_u,c,1/\beta_u}) \leq 2$ and $\gamma_u \geq 0$ in the first inequality, while in the second inequality we apply the formula of $\gamma_s$ as in Lemma \ref{lem:timedpif2}. Collect \eqref{eq:e2lambda2} and \eqref{eq:middlesecond} we get that whenever
$$t \geq \tau_3 =\tau_3(\mcH,N,a,\varepsilon) := \dfrac{N^3 e^{\pi/2}}{2}\ln \left(\dfrac{3}{\varepsilon^2} e^{4 t_0} t_0^a (\mcH_{max} - \mcH_{*})\right),$$
then the right hand side of \eqref{eq:middle} is bounded above by $\varepsilon^2/3$.

In conclusion, for large enough $t$ such that
$$t \geq \max\{t_0,\tau_1,\tau_2,\tau_3\},$$
then $F^{f_2}_{\bm{\beta},c,\mathbf{T}}(t) \leq \varepsilon^2$, and hence by \eqref{eq:tvbound} we have
$$\norm{(P^{f_2}_{\bm{\beta},c,\mathbf{T}})^t(\sigma,\cdot)-\pi^{f_2}_{\beta_t,c,1/\beta_t}}_{TV} \leq \varepsilon/2,$$
which leads to, by \eqref{eq:anneal1} and \eqref{eq:anneal2}, 
$$\mP_{\sigma}\left(\mcH\left(X^{f_2}_{\bm{\beta},c,\mathbf{T}}(t)\right) \geq \mcH_{*} + \delta\right) \leq \varepsilon.$$

\section{Applications}\label{sec:applications}

In this section, we specialize our model-independent results in Theorem \ref{thm:main} into four different models. 

In the first example, we investigate the Ising model on the complete graph $K_N$. We demonstrate rapid mixing in the low-temperature regime for the landscape-modified sampler while the critical height grows at least in $N^2$ on the original landscape that consequently leads to torpid mixing using the original Metropolis dynamics. This also serves as our warm-up example before the subsequent models with random energy landscape.

Next, we look into three models with random landscape and geometry, namely the Ising model on the random $r$-regular graph, the Ising model on the Erd\H{o}s-R\'{e}nyi random graph and finally Derrida's Random Energy Model.

In these model-dependent results, we incorporate information or estimates of the original critical height $m$ in these models, in order to obtain tighter lower bound on the relaxation time and total variation mixing time than the universal results presented in Theorem \ref{thm:main}.

\subsection{Ising model on the complete graph $K_N$}

Let $G_N = (V_N, E_N)$ be a graph with $V_N = \llbracket N \rrbracket$. For $\sigma \in \Sigma_N$, we consider the Ising model on the graph $G_N$ where the Hamiltonian function is given by
\begin{align}\label{eq:Isingcomplete}
	\mcH(\sigma) = -\dfrac{J}{2} \sum_{(v,w) \in E_N} \sigma_v \sigma_w - \dfrac{h}{2}\sum_{v \in \llbracket N \rrbracket} \sigma_v,
\end{align}
where $J > 0$ is the pairwise interaction constant and $h > 0$ is the external magnetic field. In particular, in this subsection we focus on the complete graph $G_N = K_N$. Let $\mathbf{+1}$ (resp.~$\mathbf{-1}$) denote the spin configuration with $\sigma_i = 1$ (resp.~$\sigma_i=-1$) for all $i \in \llbracket N \rrbracket$. When $h/J$ is not an integer, if we define $N^{\star} := \left\lceil\frac{1}{2}\left(N-1-\frac{h}{J}\right)\right\rceil$, then the original critical height, according to \cite[Section $1.3.3$]{DHJN17}, is 
$$m = H(\mathbf{-1},\mathbf{+1}) - \mcH(\mathbf{-1})= N^{\star}(J(N-N^{\star}) - h) = \Omega(D N^2),$$
where $D = D(J,h) > 0$ is a constant that depends on $J,h$, and we recall that from \eqref{eq:leasthighestelev} that $H(\mathbf{-1},\mathbf{+1})$ is the least highest elevation from $\mathbf{-1}$ to $\mathbf{+1}$. As the critical height grows at least in the order of $N^2$, this consequently leads to torpid relaxation and total variation mixing time of the original dynamics $X^0_{\beta}$. Also, it is clear to see in this model that $\mcH_{*} = \mcH(\mathbf{+1})$ and $\mathbf{+1}$ is the only global minimum. Using these information, we obtain the following corollary using Theorem \ref{thm:main}:

\begin{corollary}\label{cor:Isingcomplete}
	Consider the Ising model on the complete graph $K_N$ with Hamiltonian given by \eqref{eq:Isingcomplete}, where $h/J$ is not an integer. Suppose we set $c = \mcH(\mathbf{+1}) + \delta$, where $\delta$ is chosen small enough such that $0 < \delta < (\min_{\sigma \in \mathcal{LM}\cap\mathcal{GM}'}\mathcal{H}(\sigma) - \mcH_{*}) \wedge \Delta \wedge (m/4)$. In the setting of Theorem \ref{thm:main}, if we take $\varepsilon = e^{-N}$ and  
	$$\beta \geq \dfrac{3N}{\delta},$$
	then for large enough $N$, we have
	\begin{enumerate}
		\item(From torpid to rapid relaxation)\label{corit:relaxcomplete}
		$$t^0_{rel} =  \Omega\left(e^{D N^3/\delta}\right),$$
		while 
		$$t^{f_2}_{rel} = \mathcal{O}(N^3),$$
		where $D = D(J,h) >0$ is a universal constant that depends on $J,h$.
		
		\item(From torpid to rapid total variation mixing)\label{corit:mixcomplete}
		$$t^0_{mix}(e^{-N}) =  \Omega\left(e^{D N^3/\delta} N\right),$$
		while
		$$t^{f_2}_{mix}(e^{-N}) = \mathcal{O}\left(N^3 \left(\ln\left(2 e^{N}\right) + \beta \delta + \dfrac{\pi}{2} + N \ln 2\right)\right).$$
		In particular, if $$\beta = \dfrac{1}{\delta} \Theta(N),$$
		then
		$$t^{f_2}_{mix}(e^{-N}) = \mathcal{O}(N^4).$$
		
		\item(From torpid to rapid mean tunneling time to the global minimum)\label{corit:tunnelcomplete}
		$$\sup_{\sigma \in \mathcal{LM}} \mathbb{E}_{\sigma}(\tau^0_{\mathbf{+1}}) =  \Omega\left(e^{D N^3/\delta} \right).$$
		while
		$$\sup_{\sigma \in \mathcal{LM}} \mathbb{E}_{\sigma}(\tau^{f_2}_{\mathbf{+1}}) =  \mathcal{O}\left(N^3 \left(\ln 4 + \beta \delta + \dfrac{\pi}{2} + N \ln 2\right)\right).$$
		In particular, if $$\beta = \dfrac{1}{\delta} \Theta(N),$$
		then
		$$\sup_{\sigma \in \mathcal{LM}} \mathbb{E}_{\sigma}(\tau^{f_2}_{\mathbf{+1}}) = \mathcal{O}(N^4).$$
		
		\item($X^{f_2}_{\beta,c,1/\beta}$ reaches $\mcH_{*} = \mcH(\mathbf{+1})$ in polynomial time with high probability)\label{corit:reachcomplete}
		$$\mathcal{T}^{f_2}(e^{-N}) = \mathcal{O}\left(N^3 \left(\ln\left(2 e^{N}\right) + \beta \delta + \dfrac{\pi}{2} + N \ln 2\right)\right).$$
		In particular, if $$\beta = \dfrac{1}{\delta} \Theta(N),$$
		then
		$$\mathcal{T}^{f_2}(e^{-N}) = \mathcal{O}(N^4).$$
		
	\end{enumerate}
\end{corollary}

\subsection{Ising model on the random $r$-regular graph}

In this subsection, we let $r \in \mathbb{N}$ with $N > r$ and $Nr$ to be even. We say that $G_N$ is a random $r$-regular graph if we pick $G_N$ uniformly at random from the set of all simple $r$-regular graphs with $N$ vertices. A sequence of events $(A_N)_{N \in \mathbb{N}}$ is said to hold with high probability (w.h.p.) if
$$\lim_{N \to \infty} \mP(A_N) = 1.$$

Recall that the Hamiltonian function of the Ising model on $G_N$ is introduced in \eqref{eq:Isingcomplete}. It is trivial to see that in this setup $\mcH_{*} = \mcH(\mathbf{+1})$ and $\mathbf{+1}$ is the only global minimum. We now recall an upper and lower bound of the original critical height $m$ that holds w.h.p.:

\begin{lemma}\cite[Theorem $3.1$]{Dommers17}\label{lem:mestimaterregular}
	Let $G_N$ be a random $r$-regular graph with parameters $r \geq 3$ and $0 < h < C_0 \sqrt{r}$ for some small enough constant $C_0$. Then there exist universal constants $0 < C_1 < \sqrt{3}/2$ and $C_2 < \infty$ so that, w.h.p.,
	$$m_L = m_L(r,N) := (r/2 - C_1 \sqrt{r}) N\leq m \leq (r/2 + C_2 \sqrt{r}) N.$$
\end{lemma}

With the above lemma that provides an estimate of $m$, we are able to sharpen the bounds presented in Theorem \ref{thm:main} to yield the following corollary:

\begin{corollary}\label{cor:Isingrregular}
	Consider the Ising model on a random $r$-regular graph with Hamiltonian given by \eqref{eq:Isingcomplete}. Assume that the parameters $h,r$ are chosen in the same manner as Lemma \ref{lem:mestimaterregular}. Suppose we set $c = \mcH(\mathbf{+1}) + \delta$, where $\delta$ is chosen small enough such that $0 < \delta < (\min_{\sigma \in \mathcal{LM}\cap\mathcal{GM}'}\mathcal{H}(\sigma) - \mcH_{*}) \wedge \Delta \wedge (m_L/4)$. In the setting of Theorem \ref{thm:main}, if we take $\varepsilon = e^{-N}$ and  
	$$\beta \geq \dfrac{3N}{\delta},$$
	then for large enough $N$, w.h.p. the following holds:
	\begin{enumerate}
		\item(From torpid to rapid relaxation)\label{corit:relaxrregular}
		$$t^0_{rel} =  \Omega\left(e^{\frac{3 N^2}{\delta}(r/2 - C_1 \sqrt{r}) - (\ln 4) N}\right),$$
		while 
		$$t^{f_2}_{rel} = \mathcal{O}(N^3).$$
		
		\item(From torpid to rapid total variation mixing)\label{corit:mixrregular}
		$$t^0_{mix}(e^{-N}) =  \Omega\left(e^{\frac{3 N^2}{\delta}(r/2 - C_1 \sqrt{r}) - (\ln 4) N} N\right),$$
		while
		$$t^{f_2}_{mix}(e^{-N}) = \mathcal{O}\left(N^3 \left(\ln\left(2 e^{N}\right) + \beta \delta + \dfrac{\pi}{2} + N \ln 2\right)\right).$$
		In particular, if $$\beta = \dfrac{1}{\delta} \Theta(N),$$
		then
		$$t^{f_2}_{mix}(e^{-N}) = \mathcal{O}(N^4).$$
		
		\item(From torpid to rapid mean tunneling time to the global minimum)\label{corit:tunnelrregular}
		$$\sup_{\sigma \in \mathcal{LM}} \mathbb{E}_{\sigma}(\tau^0_{\mathbf{+1}};~(\mcH(\sigma))_{\sigma \in \Sigma_N}) =  \Omega\left(e^{\frac{3 N^2}{\delta}(r/2 - C_1 \sqrt{r}) - (\ln 2) N - \ln N} \right).$$
		while
		$$\sup_{\sigma \in \mathcal{LM}} \mathbb{E}_{\sigma}(\tau^{f_2}_{\mathbf{+1}};~(\mcH(\sigma))_{\sigma \in \Sigma_N}) =  \mathcal{O}\left(N^3 \left(\ln 4 + \beta \delta + \dfrac{\pi}{2} + N \ln 2\right)\right),$$
		where the expectations are taken conditional on the values of the random Hamiltonian. In particular, if $$\beta = \dfrac{1}{\delta} \Theta(N),$$
		then
		$$\sup_{\sigma \in \mathcal{LM}} \mathbb{E}_{\sigma}(\tau^{f_2}_{\mathbf{+1}};~(\mcH(\sigma))_{\sigma \in \Sigma_N}) = \mathcal{O}(N^4).$$
		
		\item\label{corit:reachrregular}
		$$\mathcal{T}^{f_2}(e^{-N}) = \mathcal{O}\left(N^3 \left(\ln\left(2 e^{N}\right) + \beta \delta + \dfrac{\pi}{2} + N \ln 2\right)\right).$$
		In particular, if $$\beta = \dfrac{1}{\delta} \Theta(N),$$
		then
		$$\mathcal{T}^{f_2}(e^{-N}) = \mathcal{O}(N^4).$$
		
	\end{enumerate}
\end{corollary}

\subsection{Ising model on the Erd\H{o}s-R\'{e}nyi random graph}

In this subsection, we investigate the Ising model on the Erd\H{o}s-R\'{e}nyi random graph, where $G_N = \mathrm{ER}_N(p)$, the Erd\H{o}s-R\'{e}nyi random graph on $N$ vertices with percolation parameter $p = g(N)/N$ and $g(N) \to \infty$ as $N \to \infty$.

First, we recall the original critical height $m$ in this setup:
\begin{lemma}\cite[Section $1.3.4$]{DHJN17}\label{lem:mestimateerdos}
	Let $G_N = \mathrm{ER}_N(p)$ be a Erd\H{o}s-R\'{e}nyi random graph with percolation parameter $p = g(N)/N$. With high probability as $N \to \infty$ we have
	$$m  \overset{N \to \infty}{\sim} \dfrac{1}{4} J N g(N).$$
\end{lemma}

Similar to the previous two models (i.e. Ising model on complete and random $r$-regular graph), with the above lemma that gives an estimate of $m$, we refine the bounds presented in Theorem \ref{thm:main} to arrive at the following corollary:

\begin{corollary}\label{cor:IsingErdos}
	Consider the Ising model on $\mathrm{ER}_N(p)$ with Hamiltonian given by \eqref{eq:Isingcomplete}. Suppose we set $c = \mcH(\mathbf{+1}) + \delta$, where $\delta$ is chosen small enough such that $0 < \delta < (\min_{\sigma \in \mathcal{LM}\cap\mathcal{GM}'}\mathcal{H}(\sigma) - \mcH_{*}) \wedge \Delta \wedge (m_L/4)$. In the setting of Theorem \ref{thm:main}, if we take $\varepsilon = e^{-N}$ and  
	$$\beta \geq \dfrac{3N}{\delta},$$
	then for large enough $N$, w.h.p. the following holds:
	\begin{enumerate}
		\item(From torpid to rapid relaxation)\label{corit:relaxErdos}
		$$t^0_{rel} =  \Omega\left(e^{\frac{3 JN^2}{4\delta}g(N) - (\ln 4) N}\right),$$
		while 
		$$t^{f_2}_{rel} = \mathcal{O}(N^3).$$
		
		\item(From torpid to rapid total variation mixing)\label{corit:mixErdos}
		$$t^0_{mix}(e^{-N}) =  \Omega\left(e^{\frac{3 JN^2}{4\delta}g(N) - (\ln 4) N} N\right),$$
		while
		$$t^{f_2}_{mix}(e^{-N}) = \mathcal{O}\left(N^3 \left(\ln\left(2 e^{N}\right) + \beta \delta + \dfrac{\pi}{2} + N \ln 2\right)\right).$$
		In particular, if $$\beta = \dfrac{1}{\delta} \Theta(N),$$
		then
		$$t^{f_2}_{mix}(e^{-N}) = \mathcal{O}(N^4).$$
		
		\item(From torpid to rapid mean tunneling time to the global minimum)\label{corit:tunnelErdos}
		$$\sup_{\sigma \in \mathcal{LM}} \mathbb{E}_{\sigma}(\tau^0_{\mathbf{+1}};~(\mcH(\sigma))_{\sigma \in \Sigma_N}) =  \Omega\left(e^{\frac{3 JN^2}{4\delta}g(N) - (\ln 2) N - \ln N} \right).$$
		while
		$$\sup_{\sigma \in \mathcal{LM}} \mathbb{E}_{\sigma}(\tau^{f_2}_{\mathbf{+1}};~(\mcH(\sigma))_{\sigma \in \Sigma_N}) =  \mathcal{O}\left(N^3 \left(\ln 4 + \beta \delta + \dfrac{\pi}{2} + N \ln 2\right)\right),$$
		where the expectations are taken conditional on the values of the random Hamiltonian. In particular, if $$\beta = \dfrac{1}{\delta} \Theta(N),$$
		then
		$$\sup_{\sigma \in \mathcal{LM}} \mathbb{E}_{\sigma}(\tau^{f_2}_{\mathbf{+1}};~(\mcH(\sigma))_{\sigma \in \Sigma_N}) = \mathcal{O}(N^4).$$
		
		\item\label{corit:reachErdos}
		$$\mathcal{T}^{f_2}(e^{-N}) = \mathcal{O}\left(N^3 \left(\ln\left(2 e^{N}\right) + \beta \delta + \dfrac{\pi}{2} + N \ln 2\right)\right).$$
		In particular, if $$\beta = \dfrac{1}{\delta} \Theta(N),$$
		then
		$$\mathcal{T}^{f_2}(e^{-N}) = \mathcal{O}(N^4).$$
		
	\end{enumerate}
\end{corollary}

\subsection{Derrida's Random Energy Model (REM)}\label{subsec:REM}

In this subsection, we consider Derrida's Random Energy Model (REM). Let $(X_{\sigma})_{\sigma \in \Sigma_N}$ be a family of i.i.d. standard normal random variables. At a spin configuration $\sigma \in \Sigma_N$, the value of the random Hamiltonian function at $\sigma$ is 
\begin{align}\label{eq:DerridaREM}
	\mathcal{H}(\sigma) = - \sqrt{N} X_{\sigma}.
\end{align}

It is known, see for instance \cite{BK04}, that the maximum of $X_{\sigma}$ over $\sigma \in \Sigma_N$, when normalized by $\sqrt{N}$, converges in probability to $\sqrt{2 \ln 2}$, that is, for any $\epsilon > 0$ we have
\begin{align}\label{eq:REMmax}
	\lim_{N \to \infty} \mathbb{P}\left( \left|\dfrac{1}{\sqrt{N}} \max_{\sigma \in \Sigma_N} X_{\sigma} - \sqrt{2 \ln 2}\right| > \epsilon \right) = 0.
\end{align}

We first present two auxiliary lemmata. In the first lemma below, we offer practical guidance on the tuning of the threshold parameter $c$ in landscape modification in Derrida's REM that leverages on \eqref{eq:REMmax}:

\begin{lemma}\label{lem:REMlem1}
	Consider Derrida's REM with Hamiltonian given by \eqref{eq:DerridaREM}. For any deterministic sequence $\epsilon_N \to 0$ as $N \to \infty$, we have 
	\begin{align}
		\lim_{N \to \infty} \mathbb{P}\left(\min_{\sigma \in \Sigma_N} \mcH(\sigma) > -N \sqrt{2 \ln 2} - \epsilon_N\right) &= 1, \label{eq:REMlem1}\\
		\lim_{N \to \infty} \mathbb{P}\left(\min_{\sigma \in \Sigma_N} \mcH(\sigma) < -N \sqrt{2 \ln 2} + \ln N\right) &= 1, \label{eq:REMlem3}\\
		\lim_{N \to \infty} \mathbb{P}\left(\min_{\sigma, \sigma^{\prime} \in \Sigma_N;~ \sigma \neq \sigma^{\prime}} |\mcH(\sigma) - \mcH(\sigma^{\prime})| \geq N^{1/4}\right) &= 1. \label{eq:REMlem2}
	\end{align}
%
	In particular, we take $\epsilon_N \equiv 0$ and set
	$$c = -N \sqrt{2 \ln 2} + \frac{N^{1/4}}{4}$$
	such that w.h.p. we have
	$$0 < c - \mcH_{*} \leq \left(\min_{\sigma \in \mathcal{LM}\cap\mathcal{GM}'}\mathcal{H}(\sigma) - \mcH_{*}\right) \wedge \Delta \wedge m/4,$$
	where we recall $m$ is the original critical height of $X^0_{\beta}$ in Derrida's REM, and $\Delta$ is introduced in Theorem \ref{thm:main}.
\end{lemma}

The proof of Lemma \ref{lem:REMlem1} is deferred to Section \ref{subsubsec:pfREMlem1}. In the second auxiliary lemma, we extract bounds on the original critical height $m$ in Derrida's REM from the relaxation time estimates in \cite[Theorem $1$]{FIKP98}. 

\begin{lemma}\label{lem:mestimateREM}
	Consider Derrida's REM with Hamiltonian given by \eqref{eq:DerridaREM}. Then there exist universal constants $C_1, C_2, C_3 < \infty$ so that, w.h.p.,
	$$N \sqrt{2 \ln 2} - C_1 \sqrt{N \ln N} \leq m \leq N \sqrt{2 \ln 2} + C_2 \sqrt{N \ln N} + C_3 \ln N.$$
\end{lemma}

The proof of Lemma \ref{lem:mestimateREM} is postponed to Section \ref{subsubsec:pfREMlem2}. With the above lemma that gives an estimate of $m$, we sharpen the bounds presented in Theorem \ref{thm:main} to arrive at the following corollary:

\begin{corollary}\label{cor:REM}
Consider Derrida's REM with Hamiltonian given by \eqref{eq:DerridaREM}. Suppose that the parameter $c$ is chosen in the same manner as Lemma \ref{lem:REMlem1}, where we set $\delta = -N \sqrt{2 \ln 2} + \frac{N^{1/4}}{4} - \mcH_{*}$. Note that w.h.p. $\delta = \Omega(N^{1/4} - \ln N)$. In the setting of Theorem \ref{thm:main}, if we take $\varepsilon = e^{-N}$ and for low enough temperature such that
		$$\beta \geq \dfrac{3N}{\delta},$$
	then for large enough $N$, w.h.p. the following holds:
	\begin{enumerate}
		\item(From torpid to rapid relaxation)\label{corit:relaxREM}
		$$t^0_{rel} =  \Omega\left(e^{\beta \left(N \sqrt{2 \ln 2} - C_1 \sqrt{N \ln N}\right) - (\ln 4) N}\right),$$
		while 
		$$t^{f_2}_{rel} = \mathcal{O}(N^3).$$
		
		\item(From torpid to rapid total variation mixing)\label{corit:mixREM}
		$$t^0_{mix}(e^{-N}) =  \Omega\left(e^{\beta\left(N \sqrt{2 \ln 2} - C_1 \sqrt{N \ln N}\right) - (\ln 4) N} N\right),$$
		while
		$$t^{f_2}_{mix}(e^{-N}) = \mathcal{O}\left(N^3 \left(\ln\left(2 e^{N}\right) + \beta \delta + \dfrac{\pi}{2} + N \ln 2\right)\right).$$
		In particular, if $$\beta = \dfrac{1}{\delta} \Theta(N),$$
		then
		$$t^{f_2}_{mix}(e^{-N}) = \mathcal{O}(N^4).$$
		
		\item(From torpid to rapid mean tunneling time to the global minimum)\label{corit:tunnelREM}
		$$\sup_{\sigma \in \mathcal{LM}} \mathbb{E}_{\sigma}(\tau^0_{\sigma_*};~(X_{\sigma})_{\sigma \in \Sigma_N}) =  \Omega\left(e^{\beta\left(N \sqrt{2 \ln 2} - C_1 \sqrt{N \ln N}\right) - (\ln 2) N - \ln N} \right).$$
		while
		$$\sup_{\sigma \in \mathcal{LM}} \mathbb{E}_{\sigma}(\tau^{f_2}_{\sigma_*};~(X_{\sigma})_{\sigma \in \Sigma_N}) =  \mathcal{O}\left(N^3 \left(\ln 4 + \beta \delta + \dfrac{\pi}{2} + N \ln 2\right)\right),$$
		where the expectations are taken conditional on the values of $(X_{\sigma})_{\sigma \in \Sigma_N}$. In particular, if $$\beta = \dfrac{1}{\delta} \Theta(N),$$
		then
		$$\sup_{\sigma \in \mathcal{LM}} \mathbb{E}_{\sigma}(\tau^{f_2}_{\sigma_*};~(X_{\sigma})_{\sigma \in \Sigma_N}) = \mathcal{O}(N^4).$$
		
		\item\label{corit:reachREM}
		$$\mathcal{T}^{f_2}(e^{-N}) = \mathcal{O}\left(N^3 \left(\ln\left(2 e^{N}\right) + \beta \delta + \dfrac{\pi}{2} + N \ln 2\right)\right).$$
		In particular, if $$\beta = \dfrac{1}{\delta} \Theta(N),$$
		then
		$$\mathcal{T}^{f_2}(e^{-N}) = \mathcal{O}(N^4).$$
		
	\end{enumerate}
\end{corollary}
 
\subsubsection{Proof of Lemma \ref{lem:REMlem1}}\label{subsubsec:pfREMlem1}
	
	First, we write $\Phi$ (resp.~ $\phi$) to be the cumulative distribution function (resp.~probability density function) of standard normal. For any sequence $(x_N)$ with $x_N \to \infty$, we recall the following asymptotics \cite[equation $(1.5.4)$]{LLR83}:
	$$\mP(X_{\sigma} > x_N) = 1-\Phi(x_N) \overset{N \to \infty}{\sim} \dfrac{\phi(x_N)}{x_N}.$$
	
	Now, we take $x_N = \sqrt{2 N \ln 2} + \dfrac{\epsilon_N}{\sqrt{N}}$ and calculate that
	\begin{align*}
		\dfrac{\phi(x_N)}{x_N} &= \dfrac{1}{\sqrt{2 N \ln 2} + \dfrac{\epsilon_N}{\sqrt{N}}} \dfrac{1}{\sqrt{2\pi}} \exp\bigg\{-\dfrac{1}{2} \left(\sqrt{2 N \ln 2} + \dfrac{\epsilon_N}{\sqrt{N}}\right)^2  \bigg\} \\
		&= 2^{-N} \exp\bigg\{ -\sqrt{2 \ln 2} \epsilon_N + \dfrac{\epsilon_N^2}{2N} - \ln\left(\sqrt{2N\ln 2}+\frac{\epsilon_N}{\sqrt{N}}\right) - \ln \sqrt{2 \pi}\bigg\}.
	\end{align*}
	
	We proceed to prove \eqref{eq:REMlem1}. Consider
	\begin{align*}
	\mathbb{P}\left(\min_{\sigma \in \Sigma_N} \mcH(\sigma) > -\sqrt{N} x_N \right) &= \mathbb{P}\left(\max_{\sigma \in \Sigma_N} X_{\sigma} < x_N\right) \\
	&= (1 - (1-\mP(X_{\sigma} > x_N)))^{2^N} \\
	&\overset{N \to \infty}{\sim} \left(1 - 2^{-N} \exp\bigg\{ -\sqrt{2 \ln 2} \epsilon_N + \dfrac{\epsilon_N^2}{2N} - \ln\left(\sqrt{2N\ln 2}+\frac{\epsilon_N}{\sqrt{N}}\right) - \ln \sqrt{2 \pi}\bigg\}\right)^{2^N} \\
	&= \lim_{N \to \infty} \exp\bigg\{-\exp\bigg\{ -\sqrt{2 \ln 2} \epsilon_N + \dfrac{\epsilon_N^2}{2N} - \ln\left(\sqrt{2N\ln 2}+\frac{\epsilon_N}{\sqrt{N}}\right) - \ln \sqrt{2 \pi}\bigg\}\bigg\},
	\end{align*}
	where the last equality utilizes \cite[Chapter 8 Lemma 1]{KK04}. Note that up to this point we have not invoked any specific properties of $\epsilon_N$. If $\epsilon_N \to 0$ as $N \to \infty$, then the limit above equals to $1$.
	
%
	
	Next, we prove \eqref{eq:REMlem3}. By taking $\epsilon_N = -\ln N$ in the proof of \eqref{eq:REMlem1} above, we see that
	\begin{align*}
		\lim_{N \to \infty} \mathbb{P}&\left(\min_{\sigma \in \Sigma_N} \mcH(\sigma) > -N \sqrt{2 \ln 2} + \ln N \right) \\
		&\quad = \lim_{N \to \infty} \exp\bigg\{-\exp\bigg\{ \sqrt{2 \ln 2} (\ln N) + \dfrac{(\ln N)^2}{2N} - \ln\left(\sqrt{2N\ln 2}+\frac{- \ln N}{\sqrt{N}}\right) - \ln \sqrt{2 \pi}\bigg\}\bigg\} = 0.
	\end{align*}

	We proceed to prove \eqref{eq:REMlem2}. Writing $Y$ to be a normal random variable with mean $0$ and variance $2$, we then have, for any $\sigma \neq \sigma^{\prime}$, 
	\begin{align*}
		\mathbb{P}\left( |\mcH(\sigma) - \mcH(\sigma^{\prime})| > N^{1/4}\right)  
		&= \mP\left(|Y| > \dfrac{1}{N^{1/4}}\right) \\
		&= 2 - 2 \Phi\left(\dfrac{1}{\sqrt{2}N^{1/4}}\right) \\
		&\overset{N \to \infty}{\sim} 1.
	\end{align*}
	
		Now, for any $\sigma \neq \sigma^{\prime}$ and $\eta \neq \eta^{\prime}$, we have
		\begin{align*}
		\mathbb{P}\left( |\mcH(\sigma) - \mcH(\sigma^{\prime})| > N^{1/4}, |\mcH(\eta) - \mcH(\eta^{\prime})| > N^{1/4}\right)  
		&= \mathbb{P}\left( |\mcH(\sigma) - \mcH(\sigma^{\prime})| > N^{1/4}\right) \\
		&\quad + \mathbb{P}\left( |\mcH(\eta) - \mcH(\eta^{\prime})| > N^{1/4}\right) \\
		&\quad - \mathbb{P}\left( |\mcH(\sigma) - \mcH(\sigma^{\prime})| > N^{1/4} \, \mathrm{or} \, |\mcH(\eta) - \mcH(\eta^{\prime})| > N^{1/4}\right)  \\
		&\overset{N \to \infty}{\sim} 1.
		\end{align*}
		
		Iterating the above process for $2^N(2^N-1)/2$ times give the desired result
		$$\mathbb{P}\left(\min_{\sigma, \sigma^{\prime} \in \Sigma_N;~ \sigma \neq \sigma^{\prime}} |\mcH(\sigma) - \mcH(\sigma^{\prime})| > N^{1/4}\right) \overset{N \to \infty}{\sim} 1.$$
	
	In particular, using \eqref{eq:REMlem3} we see that for large enough $N$ w.h.p. we obtain
	$$\mcH_{*} < -\sqrt{2\ln2}N + \ln N < c,$$
	and hence $c - \mcH_{*} > 0$. On the other hand, it is straightforward from the definition of \eqref{eq:REMlem2} that 	
	$$c - \mcH_{*} \leq \left(\min_{\sigma \in \mathcal{LM}\cap\mathcal{GM}'}\mathcal{H}(\sigma) - \mcH_{*}\right) \wedge \Delta \wedge m/4.$$
	

\subsubsection{Proof of Lemma \ref{lem:mestimateREM}}\label{subsubsec:pfREMlem2}

As in \cite[Theorem $2.1$]{HS88}, for low enough temperature we have 
$$t^0_{rel} = \Theta(e^{\beta m}).$$
The desired result follows by combining the above result with \cite[Theorem $1$]{FIKP98}, that is, w.h.p. for any inverse temperature $\beta > 0$ we have
$$\beta\left(N \sqrt{2 \ln 2} - C_1 \sqrt{N \ln N}\right) \leq \ln t^0_{rel} \leq \beta\left(N \sqrt{2 \ln 2} + C_2 \sqrt{N \ln N} + C_3 \ln N\right).$$

\section*{Acknowledgements}

The author acknowledges the financial support from the startup grant of the National University of Singapore and the Yale-NUS College.

\bibliographystyle{abbrvnat}
\bibliography{thesis}

\end{document}